\documentclass[12pt,a4]{amsart}
\usepackage{a4wide}
\usepackage{amsmath}
\usepackage{amscd}

\newtheorem{theorem}{Theorem}[section]
\newtheorem{proposition}{Proposition}[section]
\newtheorem{lemma}{Lemma}[section]
\newtheorem{corollary}{Corollary}[section]

\usepackage[utf8]{inputenc}  
\usepackage[T1]{fontenc}

\RequirePackage{times}
\usepackage{mathptmx}

\newcommand{\cF}{{\mathcal F}}

\newcommand{\cI}{{\mathcal I}}

\newcommand{\cL}{{\mathcal L}}

\newcommand{\cP}{{\mathcal P}}

\newcommand{\cW}{{\mathcal W}}

\newcommand{\cZ}{{\mathcal Z}}

\newcommand{\te}{{\theta}}

\newcommand{\Om}{{\Omega}}
\newcommand{\om}{{\omega}}

\newcommand{\Sig}{{\Sigma}}

\newcommand{\al}{{\alpha}}

\newcommand{\ka}{{\kappa}}
\newcommand{\la}{{\lambda}}
\newcommand{\La}{{\Lambda}}

\newcommand{\bbN}{{\mathbb N}}

\newcommand{\bbR}{{\mathbb R}}

\newcommand{\with}{:\;}
\newcommand{\1}{\mathbf{1}}

\begin{document}

~\vspace{0.7cm}

\begin{flushleft}
{\Large \textbf{Coupling methods for random topological Markov chains}}\\[0.5cm]

\bigskip

\begin{minipage}{.07\textwidth}
$\phantom{a}$
\end{minipage}
\begin{minipage}{.9\textwidth}
\textbf{Manuel Stadlbauer}\\
{\footnotesize Departamento de Matemática, Universidade Federal da Bahia, Av. Ademar de Barros s/n, 40170--110 Salvador, BA, Brasil. Email
manuel.stadlbauer@ufba.br}

\bigskip
\textbf{Abstract.} 
We apply coupling techniques in order to prove that the transfer operators associated with random topological Markov chains and non-stationary shift spaces with the big images and preimages-property have a spectral gap.

\bigskip

\footnotesize{\textbf{Dezember 1,  2014.} }
\end{minipage}

\end{flushleft}

\bigskip


\begin{abstract}
We apply coupling techniques in order to prove that the transfer operators associated with random topological Markov chains and non-stationary shift spaces with the big images and preimages-property have a spectral gap. 
\end{abstract}

\section{Introduction} In this note we give a further contribution to thermodynamical formalism for random countable Markov chains by applying coupling techniques and methods from optimal transport in order to circumvent problems arising from the randomized setting. In case of a  random topological Markov chain with finitely many states, it was shown by  Bogenschütz-Gundlach (\cite{BogenschutzGundlach:1995}) and Kifer (\cite{Kifer:2008}) that the method of convex cones can be adapted to the random setting in order to obtain exponential convergence to the eigenfunction of the Ruelle operator. However, even in case of deterministic topological Markov chains with countably many states, these cone techniques are not applicable and have to be replaced by proving quasi-compactness of Ruelle's operator using the method developed by Doeblin-Fortet, Ionescu Tulcea-Marinescu and Hennion (\cite{DoeblinFortet:1937,TulceaMarinescu:1950,Hennion:1993}) which is, at least so far, unavailable in the setting of random dynamical systems. In fact, basic examples without measurable families of invariant functions as in \cite[Prop. 1]{EvstigneevPirogov:2010} suggest that there might be no randomized version of the Doeblin-Fortet method. Therefore, we employ coupling techniques from the theory of Markov operators and optimal transport. It is worth noting that this approach  does not require a Lasota-Yorke condition (as, e.g., by Buzzi in \cite{Buzzi:1999}).

 In here, the setting with respect to random topological Markov chains follows closely the one in \cite{Stadlbauer:2010}, which we sketch now, postponing the details to section \ref{sec:2}. A random topological Markov chain is a random bundle transformation, that is, a commuting diagram (or fibered system)   
\[ 
 \begin{CD}
  X @>T>> X \\
  @V{\pi}VV  @VV{\pi}V \\
  \Om @>\te>> \Om, \\
 \end{CD}
\]
where $\te$ is an ergodic automorphism of the abstract probability space $(\Om,P)$ and $\pi$ is onto and measurable. In case of a random topological Markov chain, $X$ is a subset of $\bbN^\bbN\times \Om$ and each fibre $X_\om := \pi^{-1}(\{ \om \})$ has the following topological Markov structure: For a. e. $\om$, there is an associated alphabet $\cW^1_\om$ and a matrix $A_\om = (\al_{ij}^\om: i \in \cW^1_\om,j \in \cW^1_{\te\om})$, called the (random) transition matrix. Then, a pair $uv \in \cW^1_\om \times \cW^1_{\te\om}$ is called $\om$-\emph{admissible} if $ \al_{uv}^{\omega}=1$. Moreover, $(x_0, x_1,x_2 \ldots) \in X_\om$ if and only if $x_{i} \in \cW^1_{\te^i\om}$ and  $x_{i}x_{i+1}$ is $\te^i\om$-admissible for all $i= 0,1,\ldots $. 
Note that $X_\om$ comes with a natural, non-random topology defined through the shift metric on sequence spaces as defined below. With respect to this topology, $T_\om$ is uniformly expanding, $X_\om$ is closed in $\bbN^\bbN$ but in general neither compact nor locally compact.

Now assume that $\varphi:X \to \bbR$ is a measurable function such that  $\varphi_\om := \varphi \arrowvert_{X_\om}:X_\om \to \bbR$ is locally Hölder continuous almost surely. The Ruelle operator associated to $\varphi$ is then defined by, for a.e. $\om$ and $x \in X_{\te \om}$,
\[ \cL_\om(f)(x) := \sum_{y \in X_\om: T_\om(y)=x} e^{\varphi_\om (y)}f(y), \]
where $f:X_\om \to \bbR$ is in a suitable function space such that $\cL_\om$ is well-defined. 
In \cite{Stadlbauer:2010}, the notion of \emph{(random) big images and preimages} was introduced in order to have a sufficient condition for a random Ruelle's theorem at hand. This condition, inspired by the one of Sarig in \cite{Sarig:2003a}, requires that there exist subsets $\Om_{\textrm{\tiny bi}}$ and $\Om_{\textrm{\tiny bp}}$ of positive measure in the base $\Om$ and a finite subset $\cI$  of $\bbN$ such that
\begin{enumerate}
  \item for  $\te\om \in \Om_{\textrm{\tiny bi}}$ and $u \in \cW^1_{\om}$, there exists $b \in \cI$ such that $ub$ is $\om$-admissible, 
  \item for  $\te\om \in \Om_{\textrm{\tiny bp}}$ and  $u \in \cW^1_{\te\om}$, there exists $b \in \cI$ such that $bu$ is $\om$-admissible.
\end{enumerate}
The basic example for a system with this property is a random topological Markov chain such that all entries of $A_\om$ are equal to one for $\om$ in a subset  $\Om^\ast$ of $\Om$ of positive measure. In this situation, one might think of $(X,T)$ as a random topological Markov chain which is a full shift with positive probability.

The main result of \cite{Stadlbauer:2010}, Theorem 4.7 in there, is the following version of Ruelle's theorem under the assumptions of local Hölder continuity of $\varphi_\om$ and big images and preimages: There exist a random variable $\{\lambda_\om \}$, a random function $\{h_\om\}$ and a random measure $\{\mu_\om\}$, such that for almost every $\om \in \Om$, $\cL_\om(h_\om) = \lambda_\om h_{\te \om}$, $(\cL_\om)^\ast(\mu_{\te \om}) = \lambda_\om \mu_{\om}$ and, for all $f_\om \in L^1(\mu_{ \om})$,
\begin{equation} \label{eq:decay_without_speed}  \lim_{n \to \infty} \left\| ({\la_\om \cdots \la_{\te^{n-1}\om} })^{-1} {\cL_{\om}^n (f^\om)} - h_{\te^n\om}\int f_\om d\mu_\om\right\|_{L^1(\mu_{\te^n \om})} =0, 
\end{equation}
where $\cL_{\om}^n:=  \cL_{\te^{n-1}\om} \cdots \cL_{\om}$. However, the rate of convergence remained open. The main results in here, theorems \ref{theo:rpf} and \ref{theo:rpf-normalized}, give answer to this in two ways. In theorem \ref{theo:rpf}, the following generalisation of the results on   exponential decay by Bogenschütz-Gundlach (\cite{BogenschutzGundlach:1995}) and Kifer (\cite{Kifer:2008}) is obtained.
There exists $s \in (0,1)$ such that for a fibrewise Lipschitz continuous function $f=\{f_\om\}$, $n \in \bbN$ and a.e. $\om \in \Om$,
\begin{eqnarray} \label{eq:spectral_gap} 
 \left\| ({\la_\om \cdots \la_{\te^{n-1}\om} })^{-1}  {\cL_\om^{n} (f)}/{ h_{\te^{n}(\om)} } - \int f d\mu_\om \right\|_L^{\te^{n}\om}
 & \leq&   \|f \|^\om_L \cdot O_\om(s^n),
\end{eqnarray}
where $ \|f \|^\om_L$ refers to the Lipschitz norm of the restriction $f\arrowvert_{X_\om}$. Note that (\ref{eq:spectral_gap}) gives rise to an interpretation as a spectral gap: the operator might be written as a sum of a projection to a one-dimensional space and an operator which decays eventually exponentially fast, even though the waiting time to decay depends on the fibre $X_\om$. Neither the proofs in \cite{BogenschutzGundlach:1995} and \cite{Kifer:2008} nor the method used in here provide control of the random variables associated with $O_\om(s^n)$. However, as shown in theorem \ref{theo:rpf-normalized}, this is possible in case that $\cL_\om(1)=1$ a.s.. In there, the constants associated with exponential decay along returns to $\Om_{\hbox{\tiny bp}}$ are explicitly given. For illustration, these results are applied to random matrices, where it is possible to specify the associated constants under additional assumptions. Furthermore, the results are applied to decay of correlations and are used to show that the $\psi$-mixing coefficients of the associated process have exponential decay. As a further application, it is shown that the random measure given by $h_\om d\mu_\om$ is an equilibrium state.

The article's structure is the following. Section \ref{sec:22} contains the main lemma which is proven in the context of non-stationary shift spaces as introduced in \cite{ArnouxFisher:2005,Fisher:2009} in order to reveal the pathwise nature of the lemma. Namely, for a sequence of operators $(\cL_k)$ such that each $\cL_k$ maps Hölder functions defined on a Polish space $\Sig_k$ to Hölder functions on $\Sig_{k+1}$, it is shown that the dual operators eventually contract the Vasershtein distance. In particular, in a strict sense, the lemma is not related to spectral theory since even the spectra of the operators are not defined. In this situation, the technique of Birkhoff cones was applied  by Fisher in \cite{Fisher:2009} in order to characterise unique ergodicity of adic transformations through a Perron-Frobenius theorem for products of transition matrices given by a non-stationary shift space. However, this technique seems to be unavailable due to the lack of compactness of $\Sig_{k}$.
In here, the proof relies on 
an argument by Hairer and Mattingly in \cite{HairerMattingly:2008} and is based on the Monge-Kantorovich duality, optimal transport and the construction of the coupling in step (2) of the proof of lemma \ref{prop:main_estimates_1}. Even though coupling methods are known in the theory of dynamical systems and were used e.g. by Kuzmin in order to show subexponential decay for the Gauss map, by Bressaud et al. in \cite{BressaudFernandezGalves:1999} to obtain decay of correlations and the Vasershtein metric was used by Galatolo and Pacifico in \cite{GalatoloPacifico:2010} to obtain exponential decay of a  return map of a Lorenz-like flow, it seems that the  combination through ideas from optimal transport were not known in this context. 
As an application of this main lemma, it is shown in Corollary \ref{cor:non-stationary-mixing} under very mild conditions on the non-stationary shift space that invariant measures are unique and that the Ruelle operators converge.

In section \ref{sec:2}, random topological Markov chains with the b.i.p.-property are introduced. Moreover, the main lemma is adapted and applied to random topological Markov chains: the ergodicity of $\te$ allows to control the parameters in lemma \ref{prop:main_estimates_1} through successive returns to certain subsets of $\Om_{\textrm{\tiny bp}}$. In particular, if $\cL_\om(1)=1$ a.s., then this gives rise to a contraction of the Vasershtein metric  in corollaries \ref{cor:Markov_Version} and \ref{cor:Absolute_Markov_Version}. 
Section \ref{sec:3} contains two versions of Ruelle's theorem, theorems \ref{theo:rpf} and \ref{theo:rpf-normalized}, for random topological Markov chains  and their applications to random matrices, equilibrium states and asymptotic independence. For example, it is shown that the so called 
$\psi$-mixing coefficients decay exponentially fast and that there is an exponential decay of correlations of Hölder functions against functions in $L^1$. 
Finally, in section \ref{sec:4} theorems \ref{theo:rpf} and \ref{theo:rpf-normalized} are proven by combining the results of section \ref{sec:2}
with the existence of $\{ \lambda_\om \}$ and $\{\mu_\om\}$ given by  Theorem 4.7 in \cite{Stadlbauer:2010}.

\section{Non-stationary shift spaces, Vasershtein distances and the main lemma} \label{sec:22}
We now introduce the notion of non-stationary shift spaces and proof the main lemma in this context in order to reveal its pathwise character. In here, the notation is adapted to the one for random topological Markov chains and therefore differ from the definition in \cite{ArnouxFisher:2005,Fisher:2009} essentially by not explicitly using Bratteli diagrams. 
So assume that, for each  $n \in \bbN_0 := \bbN \cup \{0\}$ there is a subset $\cW^1_n \subset \bbN$ and a matrix  $A_n = \big(\al_{ij}^n,\, i \in \cW^1_n,j \in \cW^1_{n+1}\big)$, with entries $\al_{ij}^n\in\{ 0,1\}$ such that $\sum_{j \in \cW^1_{n+1}} \al_{ij}^n>0$ and $\sum_{j \in \cW^1_{n-1}} \al_{ji}^{n-1}>0$ for all $i \in \cW^1_n$ with $n \in \bbN_0$ and $n \in \bbN$, respectively. This gives rise to a sequence of shift spaces
\[
\Sig_n:=\left\{ (x_0,x_1,\ldots):\, \al_{x_i x_{i+1}}^{i+ n}=1\,\,\forall i=0,1,\ldots
\right\}
\]
and an action $T: \Sig_n \to \Sig_{n+1}$, for $n \in \bbN_0$, of the shift map $T$ defined by $T(x_0,x_1,x_2...):= (x_1,x_2...)$.
In this situation, we will refer to $((\Sig_n),T)$ as a \emph{non-stationary shift}. If $\Sig_n = \Sig_m$ for all $n,m \in \bbN_0$, then we will refer to $(\Sig_0,T)$ as a  a \emph{stationary shift}.
 
In order to define a system of neighbourhoods in analogy to stationary shifts, a finite word $a=(x_0, x_1, \ldots,x_{k-1}) \in \bbN^k$ of length $k$ will be called \emph{$n$-admissible} if $x_i \in \cW^1_{n+i}$ and $\al_{x_ix_{i+1}}^{n+i}=1$, for $i=0, \ldots , k-1$. We will refer to $\cW^k_n$ as the set  of $n$-admissible words of length $k$ and to, for $a=(a_0, a_1, \ldots,a_{k-1}) \in \cW^k_n$,  
\[ [a]_n = [a_0,a_1,...,a_{k-1}]_n :=\{ x\in \Sig_n:\, x_i=a_i,\, i=0,1,...,k-1\}\]
as a \emph{cylinder set}. As in case of a stationary shift, it follows that $\Sig_n$ is a Polish space with respect to the topology generated by these cylinder sets and that, for  $r \in (0,1)$, the  metric $d_r$ defined by
\[
d_r((x_0,x_1,x_2,\ldots),(y_0,y_1,y_2,\ldots)) := r^{\min\{i: x_i \neq y_i\} }
,\] 
is compatible with this topology. Also note that, for $a=(a_0, a_1, \ldots,a_{k-1}) \in \cW^k_n$, the restriction $T^k\arrowvert_{[a]_n} : [a]_n \to T^k([a]_n) \subset\Sig_{k+n}$ is a homeomorphism. The inverse of this map will be denoted by $\tau_a$ and it easily can be seen that
\[ \tau_a : T^k([a]_n) \to [a]_n, \; (x_0x_1 \ldots ) \mapsto (a_0a_1 \ldots a_{k-1} x_0x_1 \ldots ).  \]

\subsubsection*{Topological transitivity, mixing and the b.i.p.-property.} In contrast to stationary shift spaces, the notions of topological transitivity etc. require an additional condition which guarantees that each element a $\in \cW^1_n$ occurs infinitely often. That is, we refer to  $((\Sig_n),T)$ as \emph{topologically transitive} 
if for all
$ a \in \cW^1_{k_1}, b \in \cW^1_{k_2}$ for some $k_1,k_2 \in \bbN_0$, there exists a sequence $(n_l: l \in \bbN)$ with $n_l \nearrow \infty$ such that $b \in \cW^1_{k + n_l}$
 and 
 \[ [a]_{k_1} \cap T^{-n_l} ( [b]_{k + n_l}) \neq \emptyset  \hbox{ for all }  l \in \bbN.\] 
Furthermore, we will refer to   $((\Sig_n),T)$ as \emph{topologically mixing} if $((\Sig_n),T)$ is topologically transitive and if, for all
$ a \in \cW^1_{k_1}, b \in \cW^1_{k_2}$ for some $k_1,k_2 \in \bbN_0$, there exists $N_{ab}$ such that 
\[ [a]_{k_1} \cap T^{-n} ( [b]_{k + n}) \neq \emptyset \hbox{ for all } n \geq N_{ab} \hbox{ s.t. } b \in \cW^1_{k + n}. \]

Finally, we  will say that $((\Sig_n),T)$ has \emph{big preimages (b.p.)} if $((\Sig_n),T)$ is topologically mixing and there exists a finite subset $\cI$  of $\bbN$ and an infinite subset $\mathcal{K} \subset \bbN$ such that for all $n \in \mathcal{K} $ and  for each $a \in \cW^1_{n}$, there exists $b \in \cI \cap \cW^1_{n-1}$ with $ba \in \cW^2_{n-1}$.

\subsubsection*{Hölder continuity and the Ruelle operator.}
For  $f:\, \Sig_n \to \bbR$ and $k \in \bbN$,
\[V_k(f)=\sup \left\{ |f(x)-f(y)|: x,y \in [a]_n, a \in \cW^k_n \right\}\]
is called the $k$-th variation of $f$. We refer to $f$ as a  \emph{locally $(r,m)$-H\" older continuous function with Hölder parameter $r\in (0,1)$ and  Hölder index $m \in \bbN$} if there exists $\kappa \geq 0$ such that $V_k(f) \leq \kappa r^k$ for all $k \geq m$. If $\kappa$ is minimal, then $\kappa$ will be referred to as the \emph{Hölder constant of index $m$} of $f$. Furthermore, if $f$ is locally $(r,m)$-H\"older continuous and $\|f\|_\infty < \infty$, then we will refer to $f$ as a \emph{$r$-Hölder continuous function}. Note that $\|f\|_\infty < \infty$ and local H\"older continuity imply local H\"older continuity of index 1. If we consider $\Sig_n$ equipped with the metric $d_r$, then we also will refer to a $r$-Hölder continuous function as a \emph{Lipschitz continuous function}. In order to recover the usual definition of local H\" older continuity, it suffices to fix a metric $d_{r^\ast}$ on $\Sig_n$: the above Hölder condition is equivalent to $|f(x)-f(y)|\leq \kappa d_{r^\ast}(x,y)^s$, for all $x,y$ with $d(x,y) \leq (r^\ast)^k$ and $s = \log r / \log r^\ast$.

Now assume that $(\varphi_n:\Sig_n \to \bbR)$ is a sequence of locally $(r,l)$-H\" older continuous functions with Hölder constant $\kappa_n >0$. As in the stationary case,  the following
 estimate for Birkoff sums 
\[S_n(\varphi_k) := {\textstyle \sum_{i=0}^{n-1} \, \varphi_{i+k}\circ T^i }\]
holds.
Namely, for $n < m - l$, $x,y \in [a]_k$ with $a \in \cW_k^m$, we have that 
\begin{align} \label{eq:bounded_non_stationary_distortion}
|S_n \varphi_k (x)- S_n \varphi_k (y)|  & \leq   
 \sum_{i=0}^{n-1} \kappa_{k+i} r^{m-i} \leq r^{m-n} \log B_{k+n}, \quad  \hbox{ for } \\
 \nonumber B_{k+n} & := \exp \sum_{i=1}^{k+n} \kappa_{k+n-i} r^{i}.     
\end{align}
This estimate  allows to identify the relevant function spaces for the following sequence of operators. Assume that $(\varphi_k : k \in \bbN_0)$ has Hölder index $2$. The $k$-th Ruelle operator with respect to the potential $\varphi_k$ is  defined by, for $x \in \Sig_{k+1}$ and  $f:\Sig_{k} \to \bbR$ in a suitable space,
\[ \cL_{k}(f)(x) = \sum_{Ty=x} e^{\varphi_k(y)} f(y).\]  
For $k \in \bbN$, set $\cL^{n}_k := \cL_{k+n-1} \circ \cdots \circ \cL_{k+1}  \circ \cL_{k}$ and note that $\cL^k_{n}$ maps functions with domain $\Sig_{n}$ to functions  with domain $\Sig_{n+k}$ and that
\[ \cL^n_{k}(f)(x) = \sum_{T^n(y)=x} e^{S_n \varphi_k(y)} f(y) = \sum_{v \in \cW^n_k} e^{S_n \varphi_k(\tau_v(x))} f (\tau_v(x)).\]  
\begin{proposition} \label{prop:L_leaves_invariant_the_Lipschitz_functions}
If $\cL_{k}(1)=1$ and $(\varphi_k)$ is locally $(r,2)$-Hölder for all $k \in \bbN_0$, then $\cL^n_k$ maps $r$-Hölder continuous functions defined on $\Sig_k$ to $r$-Hölder continuous functions defined on $\Sig_{k+n}$. In particular, for a Hölder continuous function $f:\Sig_k \to \bbR$, we have $\|\cL^n_k(f)\|_\infty \leq \|f\|_\infty$ and the Hölder constant of index 1 of $\cL^n_k(f)$
is smaller than or equal to $\kappa_f r^n + \|f\|_\infty (B_{k+n} -1)$, with $\kappa_f$ referring to the Hölder constant of index {n+1} of $f$.
\end{proposition}

\begin{proof} 
The estimate (\ref{eq:bounded_non_stationary_distortion})  gives rise to, for $x,y \in \Sig_{n+k}$ with $d_r(x,y) \leq r$, 
\begin{eqnarray*}
 && |\cL^n_k(f)(x) -\cL^n_k(f)(y)| \\
& \leq & \sum_{v \in \cW^n_k} e^{S_n \varphi_k(\tau_v(x))}\left| f (\tau_v(x)) -  f (\tau_v(y))\right| \\
&  & 
+  \sum_{v \in \cW^n_k} e^{S_n \varphi_k(\tau_v(x))} \left| f (\tau_v(y)) \left(1- e^{S_n\varphi_k (\tau_v(y)) - S_n\varphi_k (\tau_v(x))}\right)\right| \\
& \leq & 
\sup_{v \in \cW^n_k} \left( \kappa_f d_r(\tau_v(x),\tau_v(y)) +    \frac{\|f\|_\infty (B_{k+n} -1)}{\log B_{k+n} }
\left|S_n\varphi_k (\tau_v(y)) - S_n\varphi_k (\tau_v(x))\right| \right) \\
& \leq & \left(\kappa_f r^n + \|f\|_\infty (B_{k+n} -1) \right) d_r(x,y),
\end{eqnarray*}
where the second step relies on $|1 - \exp(x)| \leq \exp(|x|)-1$ and the monotonicity of $(\exp(x)-1)/x$ for $x \geq 0$. 
Futhermore, $\|\cL^n_k(f)\|_\infty \leq \|f\|_\infty$ follows from the fact that $\cL^n_k$ maps positive functions to positive functions and  $\cL^n_k(1)=1$.
\end{proof} 

Before analysing the action of the dual of ${\cL^n_k}$ on Borel probability measures, we now fix the systems under consideration.   
That is, we say that $((\Sig_n),T,(\varphi_n))$ satisfies \emph{property (${\hbox{H}}^\ast$)} if $((\Sig_n),T)$ satisfies the b.p.-property, $\cL_n(1)=1$ for all $n \in \bbN_0$ and 
 the potentials $\varphi_n$ are locally $(r,2)$-H\"older continuous for all $n \in \bbN_0 \setminus \mathcal{K}$ and  locally $(r,1)$-H\"older continuous 
  for all $n \in \mathcal{K}$.

\subsection{Vasershtein distances and the coupling construction} \label{subsec:Vasershtein}
Assume that $Y$ is a Polish space and that  $\mu,\nu$ are Borel probability measures on $Y$. We then refer to 
\[\Pi(\mu,\nu) := \{m \in \cP(Y^2): \pi_1^\ast(m) = \mu, \pi_2^\ast(m) = \nu\} \]
as the couplings of $\mu$ and $\nu$, where $\cP(Y^2)$ is the set of  Borel probability measures on $Y^2$ and $\pi_i$ are the canonical projections. Now assume that $d$ is a metric compatible with the topology of $Y$. 
The Vasershtein distance of $\mu$ and $\nu$ is then defined by 
\[W(\mu,\nu) := \inf\{ {\textstyle \int d(x,y) dm} : {m \in \Pi(\mu,\nu)} \}.\]
As it is well-known (see, e.g. \cite{Villani:2009}), the Vasershtein distance is compatible with weak convergence and, by the Monge-Kantorovich duality,
\begin{equation}
\label{eq:KM-duality}
W(\mu,\nu) = \sup\{ {\textstyle \int f d\mu - \int f d\nu} : D(f) \leq 1 \},
\end{equation}
where $D(f) := \sup \{ |f(x)-f(y)|/d(x,y) : x,y \in Y\}$ refers to the \emph{Lipschitz constant} of $f$. Note that $D(f)$ is obtained, in contrast to the Hölder constant defined above, by taking the supremum over all pairs $(x,y)$.

The following proposition is the principal result of this section and is obtained by adapting the asymptotic coupling method in \cite[Section 2.1]{HairerMattingly:2008} to the non-stationary setting. Namely, as a corollary of proposition \ref{prop:L_leaves_invariant_the_Lipschitz_functions}, 
the operator defined by 
\[\int f d{\cL^n_k}^\ast(\mu) :=  \int \cL^n_k(f) d\mu\]
 maps  Borel probability measures on $\Sig_{k+n}$ to 
Borel probability measures on $\Sig_k$. By applying an asymptotic coupling method and the Monge-Kantorovich duality, it is shown below that ${\cL^n_k}^\ast$ acts as a contraction with respect to the Vasershtein distance.  
In order to shed light to the approach via the coupling construction, the involved constants are subject of a careful analysis in order to obtain effective bounds with respect to this method. 

The associated constants are defined as follows. Let $\beta \in (0,1)$, $r$ the H\"older parameter of $\{\varphi_k:\; k \in \bbN_0\}$ and $B_k$ as in (\ref{eq:bounded_non_stationary_distortion}). For $k \in \bbN_0$, define
\begin{eqnarray*} 
 \alpha_k := B_{k}/\beta, \quad
\overline{d}_k(x,y) := \min\{1, \alpha_k d_r(x,y)\}, 
\end{eqnarray*}
giving rise to the Vasershtein distances $W_k$ and Lipschitz constants $\overline{D}_k$ with respect to the metric $\overline{d}_k$ on $\Sig_k$. 
Also choose $n_k \in \bbN$ with 
\[n_k \geq \lfloor  - \log \alpha_{k} / \log r \rfloor + 1,\]
 where $\lfloor t \rfloor$ refers to the biggest integer smaller  than or equal to $t \in \bbR$. 
Furthermore, for each $k \in \bbN_0$, choose an element $\mathbf{o}_k$ in $\cW^1_k$ and define 
\begin{eqnarray*}
m_k &:= & \min\left(\left\{n \geq 1:\; k + n  \in \mathcal{K} , \,
(A_{k} A_{k+1} \cdots A_{k+n-2})_{\mathbf{o}_k j} \geq 1 \, \forall j \in \mathcal{J} \cap \cW^1_{k+n-1}
\right\}\right),
\end{eqnarray*}
which is finite by the topological mixing property. Now assume that $x \in \Sig_{k + m_k}$. By the b.p.-property, there is $j(x) \in \mathcal{J}$ such that $x \in T([j(x))]_{k+m_k-1})$. Moreover, since $\mathcal{J}$ is a finite set, there exists a finite subset $\mathcal{U}_k \subset \cW_k^{m_k -1}$ such that for each $j \in \mathcal{J} \cap \cW^1_{k+m_k -1}$, there exists an element $u = (\mathbf{o}_k u_1 \ldots j) \in \mathcal{U}_k$. Hence, for each $x \in \Sig_{k + m_k}$, there exists $u(x) \in \mathcal{U}_k$ with $x \in T^{m_k}([u(x)]_k)$. In particular, it follows that 
\begin{equation} \label{eq:bound from below - main proof - non stationary}
 C_k := \inf_{x \in \Sig_{k + m_k}}  e^{S_{m_k}\varphi_k (\tau_{u(x)}x)}  \geq  \inf_{u \in  \mathcal{U}_k} \inf_{y \in [u]_k} e^{S_{m_k}\varphi_k (y)}   > 0. 
 \end{equation}       
And finally, for $k \in \bbN_0$, set
\begin{equation} \label{def:contraction_ratio_t} t_k := \max\left\{ \beta, 1- \frac{(1 -  r^{n_k} \alpha_k)C_{k+n_k}}{B_{k+n_k}} \right\} \in (0,1). \end{equation}

\begin{lemma}\label{prop:main_estimates_1} Assume that $((\Sig_n),T,(\varphi_n))$ satisfies (${\hbox{H}}^\ast$) and that $r$ is the Hölder parameter of  $\varphi_k$. Furthermore, assume that for each $k \in \bbN_0$,  $f_k: \Sig_k \to \bbR$ is $r$-Hölder continuous and that $\mu_k$ and $\nu_k$ are Borel probability measures on $\Sig_k$. 
With  $l_k := n_k + m_{k+n}$, we then have that 
\begin{enumerate}
\item $\overline{D}_{k+n_k}(\cL_k^{n_k}(f_k)) \leq \overline{D}_k(f_k)$,
\item $\overline{D}_{k+l_k}(\cL_k^{l_k}(f_k)) \leq t_{k} \cdot \overline{D}_k(f_k)$,  
\item $W_k((\cL^{n_k}_k)^\ast(\mu_{k + n_k}),(\cL^{n_k})^\ast(\nu_{k + n_k})) \leq W_{k + n_k}(\mu_{k+n_k},\nu_{k+n_k})$,
\item $W_k((\cL^{l_k}_k)^\ast(\mu_{k+l_k}),(\cL^{l_k}_k)^\ast(\nu_{k+l_k})) \leq t_k \cdot  W_{k + l_k}(\mu_{k+l_k},\nu_{k+l_k})$.
\end{enumerate}
\end{lemma}

\begin{proof} The proof of assertions (ii) and (iv) consists of three main steps. We firstly derive a local contraction of $(\cL^n_k)^\ast$, then a concentration of $(\cL^n_k)^\ast$-images of Dirac-measures close to the diagonal and finally extend the results to arbitrary measures. The remaining, easier assertions then are proved using a simplified version of the second step.  

\subsubsection*{(1) Local contraction of $\cL$.} For a Lipschitz continuous function $f_k$, we clearly have that
\[ |f_k(x)-f_k(y)|/	\overline{d}_k(x,y)  
\leq 
\left\{\begin{array}{cl} D(f_k)/\alpha_k &: d_r(x,y) < 1/\alpha_k \\
D(f_k) &: d_r(x,y) \geq 1/\alpha_k.
 \end{array} \right.
\]
In particular, $\overline{D}_k(f_k)\leq D(f_k) \leq \alpha_k \overline{D}_k(f_k)$, which proves that $f_k$ is Lipschitz continuous on $\Sig_k$ with respect to the metric $\overline{d}_k$.
We now aiming for an estimate of $|\cL^n_\om(f)(x) -\cL^n_\om(f)(y)|$. Using $\cL_k(c)=c$, 
 we may assume without loss of generality that $\inf \{f_k(x):x \in \Sig_k\}=0$. Hence, we may assume that $\|f_k\|_\infty \leq \overline{D}_k(f_k)$.
Combining $r^{n_\om} \alpha_{\om} < 1$ with proposition \ref{prop:L_leaves_invariant_the_Lipschitz_functions} then gives that, for $x,y \in \Sig_{k + n_k}$
with $d_r(x,y) \leq r$ that 
\begin{eqnarray}
\nonumber
\frac{|\cL^{n_k}_k(f_k)(x) -\cL^{n_k}_k(f_k)(y)|}{\alpha_{k + n_k} d_r(x,y)} 
& \leq & \frac{D(f_k)r^n d_r(x,y)}{ \alpha_{k + n_k} d_r(x,y)} + \frac{\|f_k\|_\infty  (B_{k + n_k} -1)}{\alpha_{k + n_k} } \\
\nonumber
 & \leq &\overline{D}_k(f_k)\left(\alpha_k r^{n_k} +  B_{k + n_k} -1 \right)/\alpha_{k+n_k}\\
\label{eq:partial_lipschitz_coefficient}
 & \leq &\left( \overline{D}_k(f_k)B_{k+n_k}\right)/\alpha_{k+n_k}.
\end{eqnarray}
By the choice of  $\alpha_{k+n}$, the right hand side of (\ref{eq:partial_lipschitz_coefficient}) is smaller than or equal to $\beta \cdot \overline{D}_k(f_k)$. 

\subsubsection*{(2) Concentration near the diagonal.} 
We now construct $Q_{x,y} \in \Pi((\cL_k^{l_k})^\ast(\delta_x),(\cL^{l_k}_k)^\ast(\delta_y))$ with a uniform lower bound on a neighbourhood of the diagonal by considering inverse branches starting in $\mathcal{K}$ and passing through the cylinder $[\mathbf{o}]_{k+n_k} := [\mathbf{o}_{k+n_k}]_{k+n_k}$. 
Namely, the coupling is defined through the following decomposition. For $x \in \Sig_{k + l_k}$ and 
$v = (v_1v_2)\in \cW_k^{l_k}$ with $v_1 \in \cW_k^{n_k}$ and $v_2 \in \cW^{m_{k + n_k}}_{k + n_k}$,
define 
\begin{eqnarray*}
\phi_v^{(1)}(x) & := & \1_{[u(x)]_{k + n_k}} \circ T^{n_k}(x) \;\cdot \; 
\inf\left(\left\{ e^{S_{l_k}(\varphi_k)\circ\tau_{v_1 u(y)} (y)} : y \in \Sig_{k + l_k}\right\}\right) ,\\
\phi_v^{(2)}(x) &:= &e^{S_{l_k}(\varphi_k)\circ\tau_{v} (x)} - \phi_v^{(1)}(x),
\end{eqnarray*}
Observe that the indicator in the definition of $\phi_v^{(1)}(x)$ is equal to one if and only if $v_2 = u(x)$ and that $\phi_{v_1u(x)}^{(1)}(x) = \phi_{v_1u(y)}^{(1)}(y)$ for all $x,y \in  \Sig_{k + l_k}$.
 In order to obtain a bound on $\phi_v^{(1)}(x)$ for $v_2 = u(x)$, note that it follows from (\ref{eq:bound from below - main proof - non stationary}) that  
\begin{equation} 
\label{eq:lower bound potential - main proof - non stationary}
\phi_v^{(1)}(x)
\geq C_{k + n_k} \inf_{y \in [\mathbf{o}]_{k+n_k}} {e^{S_{n_k}(\varphi_k)\circ\tau_{v_1} (y)}}  
\geq \frac{C_{k + n_k}}{B_{k+n_k}}  \sup_{y \in [\mathbf{o}]_{k+n_k}}  {e^{S_{n_k}(\varphi_k)\circ\tau_{v_1} (y)}} 
\end{equation}
This gives rise to a probability measure $Q_{x,y}$ on $\Sig_k  \times \Sig_k $ by
\begin{eqnarray*}
Q_{x,y} &:=& \sum_{v_1 \in \cW_k^{n_k}}  \phi_{v_1u(x)}^{(1)}(x) \delta_{(\tau_{v_1u(x)}(x),\tau_{v_1u(y)}(y))}\\
& & +  \frac{1}{1- \sum_{v_1 \in \cW_k^{n_k}}  \phi_{v_1u(x)}^{(1)}(x)} {\sum_{v,w \in \cW_k^{l_k}}  \phi_{v}^{(2)}(x) \phi_{w}^{(2)}(y) \delta_{(\tau_{v}(x),\tau_{w}(y))}}. 
\end{eqnarray*}
As it easily can be verified, $Q_{x,y}$ is a coupling of $(\cL^{l_k}_k)^\ast(\delta_x)$ and  $(\cL^{l_k}_k)^\ast(\delta_x)$. By construction, we  have that  $\overline{d}_k(\tau_{v} (x),\tau_{v} (y)) \leq r^{n_k} \alpha_k <1$ for all $v \in \cW_k^{n_k}$ and $x,y \in [\mathbf{o}]_{k+n_k}$. Hence, by (\ref{eq:lower bound potential - main proof - non stationary}) and $\cL^{n_k}_k(1)=1$,
\[Q_{x,y} (\{(z,z'): \overline{d}_k (z,z') \leq  r^{n_k} \alpha_k \} ) \geq  \frac{C_{k + n_k}}{B_{k+n_k}}  .\]
We will now employ this estimate in order to obtain uniform contraction of $\overline{D}(f)$.
 In order to do so, let $\Delta_k :=\{(x',y') \in \Sig_k \times \Sig_k  : \overline{d}_k(x',y')\leq  r^{n_k} \alpha_k \}$. 
 For $x,y \in \Sig_{k+ l_k}$, we then have 
\begin{eqnarray*}
 && W_k ((\cL^{l_k}_k)^\ast(\delta_x),(\cL^{l_k}_k)^\ast(\delta_y)) \\
 & \leq & \int \overline{d}_k (x',y') dQ_{x,y}(x',y')
 \leq   r^{n_k} \alpha_k Q_{x,y}(\Delta_k)  + 1- Q_{x,y}(\Delta_k) \\
 & = &  1-  (1 -  r^{n_k} \alpha_k) {Q_{x,y}(\Delta_k)}
 \leq 1- \frac{(1 -  r^{n_k} \alpha_k)  C_{k + n_k}}{B_{k + n_k}} =: s_{k} <1.
 \end{eqnarray*}   
Hence, by the Monge-Kantorovich duality, 
\[
|\cL_k^{l_k}(f_k)(x) - \cL_k^{l_k}(f_k)(y)| =  \left|\int f_k d(\cL_k^{l_k})^\ast(\delta_x) - \int f_k d(\cL_k^{l_k})^\ast(\delta_y)\right| \leq \overline{D}_k(f_k) \cdot s_{k}.
\]

\subsubsection*{(3) Contraction of the Lipschitz constant and extension to arbitrary measures.} By the above,  
\[
\frac{|\cL_k^{l_k}(f)(x) - \cL_k^{l_k}(f)(y)|}{\overline{d}_{k + l_k}(x,y)} \leq 
\begin{cases}
s_{k} \cdot \overline{D}_k(f)   &:\; \overline{d}_{k + l_k}(x,y) =1 \\
\beta \cdot \overline{D}_k(f) & :\; \overline{d}_{k + l_k}(x,y) < 1. 
\end{cases}
 \]
Hence, for $t_k = \max\{ \beta, s_{k}\}$, we have $\overline{D}_{k + l_k}(\cL_k^{l_k}(f)) \leq t_{k} \cdot \overline{D}_k(f)$, which proves assertion (ii) of the proposition. By the Monge-Kantorovich duality, this implies that
\[  W_k((\cL_k^{l_k})^\ast(\delta_x),(\cL_k^{l_k})^\ast(\delta_y)) \leq t_{k} \overline{d}_{k + l_k}(x,y).\] 
Since $\mu_{k + l_k}$ and $\nu_{k + l_k}$ are probability measures on $\Sig_{k + l_k}$, there exists $Q$ in $\Pi(\mu_{k + l_k},\nu_{k + l_k})$,  referred to as optimal transport, such that  $W_{k + l_k}(\mu_{k + l_k},\nu_{k + l_k}) = \int \overline{d}_{k + l_k}(x,y) dQ(x,y)$. 
Moreover, let $P_{x,y} \in \Pi((\cL_k^{l_k})^\ast(\delta_x),(\cL_k^{l_k})^\ast(\delta_y))$ refer to an optimal transport of $(\cL_k^{l_k})^\ast(\delta_x)$ and $(\cL_k^{l_k})^\ast(\delta_y)$. Using a construction similar to the one in step (2), it is then possible to show the following. For $\varepsilon >0$ there exists $\delta > 0$ and  
 a coupling ${P}^\varepsilon_{x',y'} \in \Pi((\cL_k^{l_k})^\ast(\delta_{x'}), (\cL_k^{l_k})^\ast(\delta_{y'}))$ constructed from $P_{x,y}$ in a continuous way, such that 
 \[\left|\int \overline{d}_k \; d{P}_{x,y} - \int \overline{d}_k \; d{P}^\varepsilon_{x',y'} \right|\leq \varepsilon, \quad \forall x',y':\;  d_r(x,x'), d_r(y,y') < \delta.\]
In particular, $(x,y) \to W_k((\cL^{l_k}_k)^\ast(\delta_x),(\cL^{l_k}_k)^\ast(\delta_y))$ is continuous and there exists a locally continuous family $\{{P}^\varepsilon_{x,y} : x,y \in \Sigma_{k + l_k}\}$ which approximates the Vasershtein distance up to $\varepsilon$. For $Q^\varepsilon$ defined by $dQ^\varepsilon(x,y) := {P}^\varepsilon_{x,y} dQ(x,y)$, it is then easy to see that $Q^\varepsilon \in \Pi(\cL^{l_k}_k)^\ast(\mu_{k + l_k}),(\cL^{l_k}_k)^\ast(\nu_{k + l_k})$. Hence,
 \begin{eqnarray*}
W_k((\cL^{l_k}_k)^\ast(\mu_{k + l_k}),(\cL^{l_k}_k)^\ast(\nu_{k + l_k})) & \leq & \int \left(\int \overline{d}_k \; d{P}^\varepsilon_{x,y} \right)
 dQ(x,y) \\
 & \leq & \int W_k((\cL^{l_k}_k)^\ast(\delta_x),(\cL^{l_k}_k)^\ast(\delta_y))  dQ(x,y) + \varepsilon \\
& \leq & t_k  \int \overline{d}_{k + l_k}(x,y) dQ(x,y) + \epsilon \\ &=& t_k  W_{k + l_k}(\mu_{k + l_k},\nu_{k + l_k}) + \varepsilon.
\end{eqnarray*}
Since $\varepsilon > 0$ is arbitrary, assertion (iv) follows.

\subsubsection*{(4) Proof of assertions (i) and (iii).} The proof uses a simplified version of the arguments above. As in (2), we have  
\[
 W_k ((\cL^{n_k}_k)^\ast(\delta_x),(\cL^{n_k}_k)^\ast(\delta_y)) 
  \leq  \int \overline{d}_k (x',y') dQ_{x,y}(x',y') \leq 1.
 \]
The assertions then follow by the same argument as in (3).
\end{proof}

As a consequence of the above lemma, we obtain the following result on the convergence of $\cL_k^n$ and 
the unicity of the invariant measures for non-stationary shift spaces with property ${\hbox{H}}^\ast$. In here, we 
refer to a sequence of probability measures $(\mu_k)$ as an \emph{invariant sequence} if $\cL_k^\ast(\mu_{k+1}) = \mu_k$ for all $k \in \bbN_0$.

\begin{corollary} \label{cor:non-stationary-mixing}
Assume that $((\Sig_n),T,(\varphi_n))$ satisfies (${\hbox{H}}^\ast$), that there exists an {invariant sequence} $(\mu_k)$ of probability measures and that $n_k$ might be chosen such that $t_k$ in (\ref{def:contraction_ratio_t}) is uniformly bounded away from 1. Then the sequence  $(\mu_k)$ is unique and moreover, for a Lipschitz function $f : \Sig_k \to \bbR$, we have that  
\begin{equation} \label{eq:spectral_gap_non_stationary} \left\|\cL_k^n{f} - \int f d\mu_k \right\|_\infty \xrightarrow{n \to \infty} 0 . \end{equation}
\end{corollary}

\begin{proof} By hypothesis, $t:= \sup_n t_n <1$.  
 Lemma \ref{prop:main_estimates_1} then gives that, for $k \in \bbN_0$ and the sequence inductively defined by $p_0=k$ and $p_{j+1} = p_j + l_{p_j} $,
 \[ W_{k}((\cL_{k}^{p_j-k})^\ast(\mu_{p_j}),(\cL_{k}^{p_j-k})^\ast(\nu_{p_i})) \leq t^j W_{p_j}(\mu_{p_j},\nu_{p_j}) \leq t^j    \]
 Hence, if  $(\mu_{k})$ and $(\nu_k)$ are invariant  sequences of probability measures, then $W_k(\mu_k,\nu_k) = 0$. That is, the sequence  $(\mu_k)$ is unique. 
 Now assume that $f$ is Lipschitz and that, without loss of generality, $\overline{D}_k(f) \leq 1$. We then have by the Monge-Kantorovich duality and using the same argument as above that, for $x \in \Sig_{p_j}$,
 \begin{eqnarray*}
 \left| \cL_{k}^{p_j-k}(f)(x) - \int f d\mu_k \right| & = &  \left| \int f d(\cL_{k}^{p_j-k})^\ast (\delta_x) - \int f d(\cL_{k}^{p_j-k})^\ast(\mu_{p_j}) \right|\\
 & \leq & W_{k}((\cL_{k}^{p_j-k})^\ast(\delta_x),(\cL_{k}^{p_j-k})^\ast(\mu_{p_j})) \leq t^j.  
 \end{eqnarray*}      
Hence, $\|\cL_k^n{f} - \int f d\mu_k\|_\infty \to 0$ along a subsequence. However, since $\cL_k(1) =1$, we have that $\|\cL_k(f)\|_\infty \leq \|f\|_\infty$ for all $k$, which proves the  assertion.
\end{proof}

With respect to the corollary and its the proof, it is worth noting that the corollary is applicable in many situations since $(n_k)$ might grow arbitrarily fast. However, this implies also that the speed of convergence in (\ref{eq:spectral_gap_non_stationary}) could be arbitrarily slow.
For more specific applications of the main lemma, one might consider situations with several uniform bounds. For example, if 
$\mathcal{K}$ has a positive density and the mixing times $m_k$, the Hölder constants of $(\varphi_k)$
 and the $C_k$ in (\ref{eq:bound from below - main proof - non stationary}) are uniformly bounded (e.g. by a condition establishing some type of almost stationarity of $(\varphi_k)$ restricted to $\mathcal{U}_k$), then the main lemma would imply exponential speed of convergence to the limit in (\ref{eq:spectral_gap_non_stationary}). 
However, for random countable topological Markov chains, it is shown below that the analogues of these conditions are automatically satisfied due to the ergodicity of $\te$. Or, from an abstract point of view, the ergodicity of $\te$ establishes  a sufficient level of stationarity.

\section{Random countable topological Markov chains} \label{sec:2}
In this section, the details of the construction of random countable topological Markov chains and Hölder potentials are given. 
Note that the main difference to the definitions and arguments for non-stationary shift spaces rely in the underlying
measurable structure and therefore require adequate care.
For the definition of a random countable topological Markov chain, assume that 
$\te$ is an ergodic automorphism (i.e. ergodic, bimeasurable, invertible and probability preserving) of the probability space $(\Om,\cF,P)$, that $\{\cW^1_\om : \om \in \Om\}$ is a measurable family of subsets of $\bbN$ and that, for a.e. $\om \in \Om$,    
$A_\om=\big(\al_{ij}^\om,\, i \in \cW^1_\om,j \in \cW^1_{\te\om}\big)$ is a matrix
with entries $\al_{ij}^\om\in\{ 0,1\}$ such that $\om \mapsto A_\om$ is measurable and $\sum_{j \in \cW^1_{\te\om}} \al_{ij}^\om>0$, $\sum_{j \in \cW^1_{\te^{-1}\om}} \al_{ji}^{\te^{-1}\om}>0$ for all $i \in \cW^1_\om$.
For the random shift space 
\[
X_\om=\{ (x_0,x_1,...):\, \al_{x_ix_{i+1}}^{\te^i\om}=1\,\,\forall i=0,1,...
\},
\]
the (random) shift map $T_\om :X_\om\to X_{\te\om}$ is defined by $T_\om: (x_0,x_1,x_2...) \mapsto
(x_1,x_2,...)$. This gives rise to a globally defined map $T:X \to X$, with $X:=\{(\om,x):\, x\in X_\om\}$ and  $T(\om,x)=(\te\om,T_\om x)$. The quintuple $(X,T,\Om,P,\te)$, sometimes abbreviated by $(X,T)$, then is referred to as a \emph{random countable topological Markov chain}.

A finite word $a=(x_0, x_1, \ldots,x_{n-1}) \in \bbN^n$ of length $n$ is called \emph{$\om$-admissible}, if $x_i \in \cW^1_{\te^i \om}$ and $\al_{x_ix_{i+1}}^{\te^i\om}=1$, for $i=0, \ldots , n-1$. Moreover, $\cW^n_\om$ denotes the set  of $\om$-admissible words of length $n$ and, for $a=(a_0, a_1, \ldots,a_{n-1}) \in \bbN^n$,  
\[[a]_\om = [a_0,a_1,...,a_{n-1}]_\om :=\{ x\in X_\om:\, x_i=a_i,\, i=0,1,...,n-1\}\]
is called \emph{cylinder set}. As it easily can be verified, $X_\om$ is a closed subset of $\bbN^\bbN$ with respect to the topology generated by cylinder sets and hence, $X_\om$ is a polish space. As in case of non-stationary shift spaces,  the shift metric $d_r$ is compatible with this topology for all 
 $r \in (0,1)$. As above, $T^n_\om := T_{\te^{n-1}\om}\circ \cdots \circ T_{\te\om} \circ T_\om$ is a homeomorphism from 
 $[a]_\om$ onto $T_{\te^{n-1}\om}([a_{n-1}]_{\te^{n-1}\om})$ whose inverse will be denoted by 
 \[ \tau_a : T^n_\om([a]_\om) \to [a]_\om, \quad (x_0 x_1 \ldots) \mapsto (a x_0 x_1 \ldots). \]
 The set of those $\om \in \Om$ where the cylinder is nonempty will be denoted by $\Om_a$, that is 
\[\Om_a =\{\om \with [a]_\om\ne\emptyset\} = \{\om: a \in \cW^n_\om\}.\]
The set $\cW^n$ refers to the set of words $a$ of length $n$ with $P(\Om_a)>0$. 

For the definition of random probability measures, we will adapt the definition in 
\cite{Crauel:2002} to our setting. This adaption relies on the fact that the measurable structure of $\{(x,\om) : x \in X_\om, \om \in \Om\}$ is induced by the product structure of $\bbN^\bbN \times \Om$. 
That is, we refer to $\mu = \{\mu_\om\}$ as a random probability measure, if $\mu$ is a map
\[ \mu : \mathcal{B} \times \Om \to [0,1], \quad (B,\om) \mapsto \mu_\om(B), \]
where $\mathcal{B}$ refers to the Borel $\sigma$-algebra of the countable full shift $\bbN^\bbN$, such that
\begin{enumerate}
  \item for every $B \in \mathcal{B}$, $\om \mapsto \mu_\om(B)$ is measurable,
  \item for $P$-almost every $\om \in \Om$, $B \mapsto \mu_\om(B)$ is a probability measure with support $X_\om$.
\end{enumerate}

\subsubsection*{The b.i.p.-property.}
With the notion of cylinders at hand, we now give the definitions of topologically mixing and big images and preimages. We say that $(X,T)$ is \emph{(fibrewise) topologically mixing} if for all
$a,b\in\cW^1$, there exists a $\bbN$-valued random variable $N_{ab}=
N_{ab}(\om)$ such that, for $n\geq N_{ab}(\om)$, $a \leq \cW^1_\om$ and
$\te^n\om \in \Om_b$, we have that $[a]_\om\cap(T^n_\om)^{-1}[b]_{\te^n\om}\ne \emptyset$. 

Morever, assume that there exist $\Om_{\textrm{\tiny bi}}\subset \Om$ and $\Om_{\textrm{\tiny bp}}\subset \Om$ of positive measure and a finite subset $\cI$  of $\bbN$ such that 
\begin{enumerate}
  \item for each $\om \in \Om_{\textrm{\tiny bi}}$ and $a \in \cW^1_{\te^{-1}\om}$, there exists $b \in \cI$ with $ab \in \cW^2_{\te^{-1}\om}$
  \item for each $\om \in \Om_{\textrm{\tiny bp}}$ and  $a \in \cW^1_\om$, there exists $b \in \cI$ with $ba \in \cW^2_{\te^{-1}\om}$.
\end{enumerate}
If, in addition, $(X,T)$ is topologically mixing, then $(X,T)$ is said to have the  \emph{(relative) big images and big preimages property} or \emph{(relative) b.i.p.-property}. Note that this definition is equivalent to one in \cite{Stadlbauer:2010}.

\subsubsection*{Hölder and Lipschitz continuity.} 
In order to define fibrewise Hölder continuity, the $n$-th variation of a function 
$f:\, X\to\bbR$, $(\om,x) \mapsto f_\om(x)$ is defined by
\[V^\om_n(f)=\sup\{ |f_\om(x)-f_\om(y)|:\, x_i=y_i,\, i=0,1,\ldots,
n-1\}.\]
We refer to $f$ as a \emph{(fibrewise) locally $(r,k)$-H\" older continuous function with parameter $r \in (0,1)$ and index $k\in \bbN \cup \{0\}$} if there exists a random variable $\ka=\ka(\om)\geq 0$ such  that $\int \log \ka dP<\infty$ and
 $V^\om_n(f)\leq\ka(\om)r^n$ for all $n\geq k$. Furthermore, we will refer to a locally H\" older continuous function with $\|f_\om\|_\infty < \infty$ a.s. as a \emph{(fibrewise) Hölder continuous function}. Note that a Hölder continuous function is Lipschitz with respect to $d_r$. Therefore, we refer to 
 \[ D_\om(f) := \sup \{(f(x) - f(y))/d_r(x,y) : x,y \in X_\om\}. \] 
as the \emph{(relative) Lipschitz constant} of $f$.
In  complete analogy to non-stationary shift spaces, there is the following estimate for Birkhoff sums $S_n f_\om:= {\textstyle \sum_{k=0}^{n-1} \, f_{\te^k\om}\circ T^k_\om }$ (see also \cite[p. 80]{Stadlbauer:2010}).   For $n \leq m$, $x,y \in [a]_\om$ with $a \in \cW_\om^m$, and a locally $(r,m-n+1)$-Hölder continuous function $f$ 
\[
|S_nf_\om (x)- S_n f_\om (y)|  \leq r^{m-n}\sum_{k=1}^\infty \ka(\te^{n-k}\om) r^k.  
\]
It follows from an application of the ergodic theorem, that $\lim_{k \to \infty} k^{-1}  \log \ka(\te^{n-k}\om) = 0$ a.s. Hence, the radius of convergence of the power series on the right hand side is equal to 1 and, in particular, it follows from $r<1$ that the right hand side is finite a.s (see, e.g. \cite{DenkerKiferStadlbauer:2008}). 
However, if also $\int \ka dP<\infty$, e.g. if condition (H) below holds, monotone convergence and $\te$-invariance of $P$ imply that
\[ \int \sum_{k=1}^\infty \ka(\te^{n-k}\om) r^k dP(\om) = \sum_{k=1}^\infty r^k \int \ka dP  = \frac{r}{1-r} \int \ka dP < \infty . \]
In particular, for a given locally H\"older continuous function $\varphi$ with index less than or equal to $(m-n+1)$, we hence have that  
\begin{eqnarray}
\label{eq:key} (B_{\te^n \om})^{-1} \leq  (B_{\te^n \om})^{-r^{m-n}} \leq e^{S_n\varphi_\om (x)- S_n \varphi_\om (y)} \leq  (B_{\te^n \om})^{r^{m-n}} \leq  B_{\te^n \om}, \\
\nonumber \int \log B_\om dP< \infty, \quad \hbox{ where }B_\om:= \exp \sum_{k=1}^\infty \kappa(\te^{-k}\om)r^k.
\end{eqnarray}
After these considerations, we are now in position to give the definition of the systems under consideration. That is, we consider  $(X,T,\Om,P,\te)$ with the b.i.p.-property and a potential function $\varphi$ satisfying the following assumptions (H) and (S) on Hölder continuity and summability, respectively.
\begin{itemize}
  \item[{(H)}] The potential $\varphi$ is locally Hölder continuous with index 2 and the associated random variable $\kappa$ satisfies $\int \ka dP<\infty$. Furthermore, $V_1^\om(\varphi) < \infty$ for a.e. $\om \in \te^{-1}(\Om_{\textrm{\tiny bi}} \cup \Om_{\textrm{\tiny bp}})$.
  \item[{(S)}]  $\int |\log \cL_{\te^{-1}\om}(1) | dP(\om) < \infty$.
\end{itemize}
If  $\cL_\om (1)=1$ a.s., then the above condition (H) can be replaced by ($\hat{\hbox{H}}$) below, which differs only by integrability of the H\"older constant. 
\begin{itemize}
  \item[($\hat{\hbox{H}}$)] The potential $\varphi$ is locally H\"older continuous with index 2  and  $V_1^\om(\varphi) < \infty$ for a.e. $\om \in \te^{-1}(\Om_{\textrm{\tiny bi}} \cup \Om_{\textrm{\tiny bp}})$.
\end{itemize}
Note that by definition of local H\"older continuity, condition ($\hat{\hbox{H}}$) implies that the associated random variable $\kappa$ satisfies $\int \log \ka dP<\infty$. The reason behind that slightly more general integrability condition stems from the random variable $\lambda$ in theorem
\ref{theo:rpf} (see also equation (\ref{eq:spectral_gap})). If $\cL_\om (1)=1$, then automatically, $\lambda=1$ a.s. However, if 
 if  $\cL_\om (1) \neq 1$ on a set of positive measure, the first step in the construction of $\lambda$ is the construction of the pressure $P_G(\varphi)$ (as defined below) which depends on the almost subadditive ergodic theorem and therefore requires that $\log B_\om$ is integrable (see \cite{DenkerKiferStadlbauer:2008,Stadlbauer:2010}). 

\subsection{Contraction of Vasershtein distances} \label{sec:Vasershtein}
In order to apply lemma \ref{prop:main_estimates_1}, we now adapt the relevant variables and deduce their finiteness from the structure of $(X,T)$, where we assume that $(X,T)$ satisfies the b.i.p.-property.
For $\beta \in (0,1)$, $r$ given by the H\"older continuity of $\{\varphi_\om:\; \om \in \Om\}$ and $B_\om$ as in (\ref{eq:key}), define
\begin{eqnarray*} 
 \alpha_\om := B_{\om}/\beta, \quad \overline{d}_\om(x,y) := \min\{1, \alpha_\om d_r(x,y)\} 
\end{eqnarray*}
and let 
$n_\om$ be a $\bbN$-valued random variable such that $P$-a.s.,
\[n_\om \geq \lfloor  - \log \alpha_{\om} / \log r \rfloor + 1.\]
Furthermore, assume that $\mathbf{o}_\om $ is an $\bbN$-valued random variable such that $\mathbf{o}_\om \in \cW^1_\om$ (e.g., set $\mathbf{o}_\om := \min (\cW^1_\om)$). Note that, since $(X,T)$ is topologically mixing and $\te$ is ergodic,  
\begin{eqnarray*}
m_\om &:= & \min(\{n \geq 1:\; \te^{n}\om \in \Om_{\textrm{\tiny bp}}, \,
(A_{\om} A_{\te\om} \cdots A_{\te^{n-1}\om})_{\mathbf{o}_\om  j}\geq 1\, \forall j \in \mathcal{J} \cap \cW^1_{\te^{n-1}\om}
\}),
\end{eqnarray*}
is a.s. finite. Moreover, the big preimage property gives rise to the following argument. For 
each $x \in X_{\te^{m_\om}(\om)}$, there exists $u(x)=(u_0\ldots u_{m_\om -1}) \in \cW^{m_\om}$ such that $u_0= \mathbf{o}_\omega$, $u_{m_\om -1} \in \cI$ and $x \in T^{m_\om}_\om([u(x)]_\om)$. By choosing $u(x)$ minimal according to the lexicographic order, we obtain a finite and measurable family $\{u(x) : x \in   X_{\te^{m_\om}(\om)}\}$. In particular, 
 \begin{eqnarray}  \label{eq:bound from below - main proof}
 C_\om &:=& \inf_{x \in X_{\te^{m_\om}(\om)}}  e^{S_{m_\om}\varphi_{\om} (\tau_{u(x)}x)} >0 \quad \hbox{ and }\\
 \nonumber
 t_\om &:=& \max\left\{ \beta, 1- \frac{(1 -  r^{n_\om} \alpha_\om)C_{\te^{n_\om}(\om)}}{B_{\te^{n_\om}(\om)}} \right\} \in (0,1) \quad
  P\hbox{-a.s.}. 
 \end{eqnarray}       
 For $\om \in \Om$, we refer to  $W_\om$ as the Vasershtein distance with respect to the metric $\overline{d}_\om$ and to $\overline{D}_\om(f)$ as the Lipschitz constant with respect to $\overline{d}_\om$. In the context of random topological Markov chains, lemma \ref{prop:main_estimates_1}
is as follows.

\begin{lemma}\label{prop:main_estimates} Assume that ($\hat{\hbox{H}}$) holds, 
 that $f =\{f_\om\}$ is Hölder continuous with respect to the same parameter as $\varphi$ and that  $\mu =\{\mu_\om\}$ and $\nu =\{\nu_\om\}$ are random probability measures. 
For a.e. $\om \in \Om$, with $k_\om := n_\om + m_{\te^{n_\om}(\om)}$, we then have that 
\begin{enumerate}
\item $\overline{D}_{\te^{n_\om} (\om)}(\cL_\om^{n_\om}(f)) \leq \overline{D}_\om(f)$,
\item $\overline{D}_{\te^{k_\om} (\om)}(\cL_\om^{k_\om}(f)) \leq t_{\om} \cdot \overline{D}_\om(f)$,  
\item $W_\om((\cL^{n_\om})^\ast(\mu_{\te^{n_\om}(\om)}),(\cL^{n_\om})^\ast(\nu_{\te^{n_\om}(\om)})) \leq W_{\te^{n_\om}(\om)}(\mu_{\te^{n_\om}(\om)},\nu_{\te^{n_\om}(\om)})$,
\item $W_\om((\cL^{k_\om})^\ast(\mu_{\te^{k_\om}(\om)}),(\cL^{k_\om})^\ast(\nu_{\te^{k_\om}(\om)})) \leq t_\om \cdot  W_{\te^{k_\om}(\om)}(\mu_{\te^{k_\om}(\om)},\nu_{\te^{k_\om}(\om)})$.
\end{enumerate}
\end{lemma}

\noindent\textbf{Remark.}
 The pathwise character of lemma \ref{prop:main_estimates_1} admits in contrast to spectral theoretic methods like in \cite{DoeblinFortet:1937,TulceaMarinescu:1950,Hennion:1993} an immediate application in a fibered setting. Moreover, it is worth noting that a recent result of Zhang in \cite{Zhang:2013} establishes a fibrewise Monge-Kantorovich duality, which would allow, even though not necessary for the proof, to choose the couplings $Q_{x,y}$ and $Q$ in step (3) in a measurable way.

As an application of the proposition above, that is, through construction of $n_\om$ and $m_\om$, we obtain the key observation  of this note, giving rise to exponential decay of the Vasershtein distance $W$ with respect to the initial shift metric $d_r$. For completeness, the dual statement on decay of Lipschitz constants  also is included.

\begin{corollary} \label{cor:Markov_Version} Assume that $(X,T,\varphi)$ satisfies the b.i.p.-property, ($\hat{\hbox{H}}$) holds and that $\cL_\om(1)=1$ a.s. Then there exist constants $t \in (0,1)$, $c \in (0,\infty)$ and random sequences $(k_n(\om):n \in \bbN)$ and $(l_n(\om):n \in \bbN)$ such that 
\begin{enumerate}
\item for random probability measures $\mu, \nu$ and $n \in \bbN$, with $\mu^n_\om := (\cL_\om^n)^\ast (\mu_{\te^n \om})$, 
\begin{eqnarray*} 
 W( \mu^{k_n(\om)}_{\te^{-k_n(\om)}(\om)}, \nu^{k_n(\om)}_{\te^{-k_n(\om)}(\om)})
& \leq &   2B_\om \cdot t^n W(\mu_\om,\nu_\om), \\ 
 W ( \mu^{l_n(\om)}_{\om}, \nu^{l_n(\om)}_{\om})  
 &\leq &  c \cdot t^n W (\mu_{\te^{l_n(\om)}(\om)},\nu_{\te^{l_n(\om)}(\om)}),
\end{eqnarray*}
\item for a fibrewise Lipschitz continuous function $f$ and  $n \in \bbN$,
\begin{eqnarray*} 
{D}(\cL_{\te^{-k_n(\om)}(\om)}^{k_n(\om)}(f)) & \leq & 2B_\om \cdot
 t^n {D}(f_{\te^{-k_n(\om)}(\om)}),\\
 {D}(\cL_\om^{l_n(\om)}(f)) & \leq &   c \cdot t^n {D}(f_\om).
\end{eqnarray*}
\end{enumerate} 
If $B_\om$ is uniformly bounded, then the above is satisfied for $c:= 2\, \hbox{\rm ess-sup}_{\om\in \Om} (B_\om)$.
\end{corollary}

\begin{proof} We begin with the construction of $(l_n(\om))$. In order to apply lemma \ref{prop:main_estimates}, choose $\beta =1/2$ and $C>0$ and $B\geq 1$ such that $P(\Om_{B,C})>0$, where $\Om_{B,C} := \{\om \in \Om : B_\om \leq B,  C_\om \geq C\}$. Note that this implies that $\alpha_\om = 2 B_\om$.
 For 
\[n_\om := \min\left(\left\{ n \in \bbN :\; \te^n(\om) \in \Om_{B,C}, n \geq  \lfloor -\log(2 \alpha_\om)/\log r\rfloor +1 \right\}\right),\]
it follows that $t:= 1- C/2B \geq t_\om$ for $P$-a.e. $\om$.  
The sequence $l_n(\om)$ is now defined as follows. For $n=0$, set $l_0(\om):= n_\om$ and define $l_n$ by 
\[l_n(\om) := l_{n-1}(\om) + m_{\te^{l_{n-1}(\om)}(\om)}  + 
n_{\te^{l^\ast(n-1,\om)}(\om)},
\]
where  $ {l^\ast}(n-1,\om) = l_{n-1}(\om) + m_{\te^{l_{n-1}(\om)}(\om)}$. Hence, $\te^{l_n(\om)}(\om)\in \Om_{B,C}$ and $\alpha_{\te^{l_n(\om)}(\om)} \leq 2B := c$.  
In order to obtain $(l_n(\om))$, we adapt the above to negative powers of $\te$. That is, with  $\widetilde{\Om}_{B,C} := \{\te^{m_\om} : \om \in \Om_{B,C}\}$, define 
\begin{eqnarray*}
\widetilde{n}_\om &:= &  \min\left(\left\{ n \in \bbN :\; \te^{-n}(\om) \in \widetilde{\Om}_{B,C}, n \geq  \lfloor -\log(2 \alpha_{\te^{-n}(\om)})/\log r\rfloor +1 \right\}\right),\\
\widetilde{m}_\om &:= & \min\left(\left\{n : \te^{-n}\om \in \Om_{B,C}, n = m_{\te^{-n}(\om)}  \right\} \right), 
\hbox{ for } \om \in \widetilde{\Om}_{B,C}, \\
k_0(\om) &:= &  \widetilde{n}_\om + \widetilde{m}_{\te^{-\widetilde{n}_\om}(\om)}; \quad  k_n(\om) := k_{n-1}(\om) + \widetilde{n}_{\te^{-k_{n-1}(\om)}(\om)}  +  \widetilde{m}_{\te^{-k^\ast(n-1,\om)}(\om)}, \\
& & \hbox{with }   k^\ast(n-1,\om) = k_{n-1}(\om) + \widetilde{n}_{\te^{-k_{n-1}(\om)}(\om)}.
\end{eqnarray*}
The assertions of the corollary  follow from lemma \ref{prop:main_estimates} and 
$\alpha_\om d \geq \overline{d}_\om \geq d$
\end{proof}

For the statement of the corollary, the exponential decay was formulated with respect to the sequences $(k_n)$ and $(l_n)$ since this approach allows to explicitly construct the involved constants, which might be useful for questions concerning stochastic stability (as e.g. in \cite{ShenStrien:2013}). However, a slight modification of the construction allows to obtain a more classical formulation of exponential decay. 
For ease of exposition, the result only contains the statement with respect to the Vasershtein distance.
 
\begin{corollary} \label{cor:Absolute_Markov_Version} There exists $s \in (0,1)$ and a positive random variable $c_\om^\ast$ such that, for each pair of random probability measures $\mu, \nu$ and $n \in \bbN$, we have 
\[  W( \mu^{n}_{\te^{-n}(\om)}, \nu^{n}_{\te^{-n}(\om)})
 \leq   c_\om^\ast \cdot s^n W(\mu_\om,\nu_\om) \hbox{ and }
 W ( \mu^{n}_{\om}, \nu^{n}_{\om})  \leq
 c_\om^\ast \cdot s^n W(\mu_{\te^n(\om)},\nu_{\te^{n}(\om)}). \]
\end{corollary}

\begin{proof} Similar to the proof of corollary \ref{cor:Markov_Version}, choose $B,C,M$ such that 
\[\Om_{B,C,M} := \{\om \in  \Om : B_\om \leq B,  C_\om \geq C, m_\om \leq M, B_{\te^{m_\om} \om} \leq B\} \]
has positive measure. Furthermore, set $K:=\lfloor -\log(2 B)/\log r\rfloor +1$ and let $\eta_k^\om$ 
refer to the $k$-th entrance time to $\Om_{B,C,M}$, that is
\[ \eta^\om_k := \min\left(\left\{ n_{k} \in \bbN :\; 
\exists \; 0 < n_1 < \cdots < n_{k} \hbox{ s.t. } \te^{n_i}(\om) \in \Om_{B,C,M}  \hbox{ for } i=1, \ldots k 
 \right\}\right).\]
Since $\eta^\om_{k} \geq k $ for all $k \in \bbN$, it follows that for a.e. $\om \in \Om$, $n_{\te^{m_\om}(\om)} := \eta^\om_{M+K} - m_\om \geq K$. Hence, lemma \ref{prop:main_estimates} is applicable to each transition from $l_k(\om)$ to $l_{k+1}(\om)$, where $l_k(\om)$ is defined by $l_k(\om):= \eta^\om_{(M+K)k} + m_{\te^{\eta^\om_{(M+K)k}}\om}$. For $s= 1- C/2B$ as in proof above, it follows from the comparability of $d$ and $\overline{d}_\om$, that for a.e. $\om$,  
\[ W((\cL_{\te^{m_\om}(\om)}^{l_k(\om) - m_\om})^\ast \mu_{\te^{l_k(\om)}\om},(\cL_{\te^{m_\om}(\om)}^{l_k(\om) - m_\om})^\ast\nu_{\te^{l_k(\om)}\om}) \leq   2B s^k W(\mu_{\te^{l_k(\om)}\om}, \nu_{\te^{l_k(\om)}\om}).\]  
In order to obtain a rate with respect to $n$ as in the statement of the corollary, choose $k$ such that $l_k \leq n <l_{k+1}$. By applying the ergodic theorem twice, it follows that 
\[ \frac{1}{n}\log s^k = \frac{l_k}{n} \frac{k}{l_k} \log s \xrightarrow{k,n \to \infty} \frac{1}{(M+K)P(\Om_{B,C,M})} \log s.  \]
Moreover, note that substituting $k$ with $k+1$ does not change the limit.  
Hence, for each $t \in (0,1)$ with $\log t > (\log s)/((M+K)P(\Om_{B,C,M}))$, there is a random variable $c^\ast$ such that 
\[ W((\cL_{\om}^{\te^n\om})^\ast \mu_{\te^{n}\om},(\cL_{\om}^{\te^n\om})^\ast \nu_{\te^{n}\om}) 
\leq   c^\ast_\om t^n W(\mu_{\te^{n}\om}, \nu_{\te^{n}\om}).
\] 
The remaining assertion follows from the same arguments by considering the $(M+K)$-th entrances to $\Om_{B,C,M}$ with respect to $\te^{-1}$ instead of $\te$. 
\end{proof}

\section{The random version of Ruelle's theorem} \label{sec:3}
We now apply the above results in order to obtain exponential decay in Ruelle's theorem. We first recall the basic results from \cite{Stadlbauer:2010} for $(X,T,\Om,P,\te)$ with the b.i.p.-property and conditions (H) and (S). For a given $a \in \cW^1$ and a measurable family $\{\xi_\om \in [a]_\om : \om \in \Om\}$, the \emph{$n$-th local preimage function} is defined by 
\[\cZ_n^\om(a) :=  \cL_{\om}^n(1_{[a]})(\xi_{\te^n\om}),\quad \forall \om \in \Om_a \cap \te^{-n}(\Om_a).  \]
The \emph{relative Gurevi\v{c} pressure} $P_G(\phi)$ is then defined as follows. Due to the restriction on $\om$ and $n$ in the definition of the local preimage function, the classical definition of pressure has to be modified as a limit along a subsequence given by returns with respect to the base. More precisely,  
for $\Om' \subset \Om$ of positive measure and $\om \in \Om$, define $J_\om({\Om'}) := \{ n \in \bbN \with \te^n\om \in \Om' \}$. For $N\in \bbN$ such that $\Om^\ast:= \{\om \in \Om_a \with N_{aa}(\om) \leq N \}$ is a set of positive measure, the Gurevic pressure of $(X,T,\varphi)$ is then defined by
\begin{equation} \label{def:gurevic_via_sup}
P_G(\varphi) := \lim_{{n \to \infty, n \in J_\om({\Om^\ast})}} \frac{1}{n} \log \cZ_n^\om(a).
\end{equation}
By Theorem 3.2 in \cite{DenkerKiferStadlbauer:2008}, the above limit exists, is finite, $P$-almost surely constant for $\om \in \Om_a$ and independent of the choices of $a$ and $N$. 
Combining  Theorem 4.1 in \cite{Stadlbauer:2010} with Corollaries \ref{cor:Markov_Version} and \ref{cor:Absolute_Markov_Version} above then allows to obtain the main result of this note. In order to state the theorem, we define the following random Lipschitz norm with respect to $d_r$, where $r$ is given by the Hölder parameter of $\varphi$. For $\om \in \Om$ and $f_\om : X_\om  \to \bbR$ with $D_\om(f) < \infty$, set 
\[ \|f_\om\|_L^\om := D_\om(f_\om) + \| f_\om\|_\infty.\] 
Furthermore, for the statement of the theorems, let  $c$, $c^\ast_\om$ be as in corollaries   \ref{cor:Markov_Version} and \ref{cor:Absolute_Markov_Version} and, for a given random variable $\la:\Om \to (0,\infty)$, define $\La_k(\om):= \prod_{i=0}^{k-1}\la_{\te^i(\om)}$.
\begin{theorem} \label{theo:rpf}  Assume that $(X,T,\phi)$ satisfies the b.i.p.-property and (H) and (S) hold. Then there exist a random variable $\lambda:\Om \to \bbR$ with $\int \log \lambda_\om dP = P_G(\varphi)$, a measurable family of functions $\{h_\om : \om \in \Om\}$ and a random probability measure $\{\mu_\om: \om \in \Om\}$ such that the following holds.
\begin{enumerate}
\item For a.e. $\om \in \Om$, $h_\om:X_\om \to \bbR$ is a strictly positive function satisfying $\cL_\om  h_\om=\la_\om h_{\te\om}$ and  $\int h_\om d \mu_\om=1$.
\item $\{\log h_\om\}$ is H\"older continuous with H\"older constant of index 1 bounded by $B_\om-1$ and the same Hölder parameter as $\varphi$.
\item For a.e. $\om \in \Om$, $\cL_\om^\ast(\mu_{\te\om}) = \la_\om\mu_\om$.
\item The probability measure given by $h_\om d\mu_\om dP$ is $T$-invariant and ergodic.
\item $\{\mu_\om\}$ is the unique random probability measure with $\cL_\om^\ast(\mu_{\te\om}) = \la_\om \mu_\om$ a.s. Furthermore, $\{h_\om \}$ 
is, up to scalar multiplication, the unique positive, measurable and non-trivial function with  $\cL_\om  h_\om= \la_\om h_{\te\om}$ a.s.
\item There exists $t \in (0,1)$, a positive random variable $K$ and a random sequence $(l_n(\om):n \in \bbN)$ such that for each fibrewise Lipschitz continuous function $f=\{f_\om\}$, $n \in \bbN$ and a.e. $\om \in \Om$,
\begin{eqnarray*} 
 \left\| \frac{\cL_\om^{l_n(\om)} (f_\om)}{ \La_{l_n}(\om) h_{\te^{l_n}(\om)} } - \int f   d\mu_\om \right\|_L^{\te^{l_n}\om}
 & \leq &  c K_\om \cdot t^n \|f \|^\om_L.
\end{eqnarray*}
\item There exists $s \in (0,1)$ such that for each fibrewise Lipschitz continuous function $f=\{f_\om\}$, $n \in \bbN$ and a.e. $\om \in \Om$,
\begin{eqnarray*} 
 \left\| \frac{\cL_\om^{n} (f)}{ \La_{n}(\om) h_{\te^{n}(\om)} } - \int f d\mu_\om \right\|_L^{\te^{n}\om}
 & \leq &  c^\ast_\om  K_\om \cdot s^n \|f \|^\om_L.
\end{eqnarray*} 
\end{enumerate} 
\end{theorem}

The proof of the theorem makes use of a normalised version of Ruelle's operator, which is necessary in order to obtain contraction of the  Vasershtein metric. That is, with $\{h_\om\}$ and $\lambda$ as in the theorem, let $\tilde{\cL}_\om$  refer to the Ruelle operator with respect to the potential 
\[\tilde{\varphi}_\om (x):=  \varphi_\om(x) + \log h_\om(x) - \log h_{\te\om}(T_\om(x)) - \log \lambda_\om.\]
As an application of corollaries \ref{cor:Markov_Version}  
 and \ref{cor:Absolute_Markov_Version} we obtain the following theorem.

 \begin{theorem} \label{theo:rpf-normalized}  Assume that $(X,T,\phi)$ satisfies the b.i.p.-property and (H) and (S) hold. Then there exist a random probability measure $\{\nu_\om: \om \in \Om\}$, constants $t \in (0,1)$, $c \in (0,\infty)$ and random sequences $(k_n(\om):n \in \bbN)$ and $(l_n(\om):n \in \bbN)$ such that for each fibrewise Lipschitz continuous function $f=\{f_\om\}$, $n \in \bbN$ and a.e. $\om \in \Om$,
\begin{eqnarray*}
\left\|\tilde{\cL}_\om^{l_n(\om)} (f) - \int f   d\nu_\om \right\|_L^{\te^{l_n}\om}
 & \leq &  2 c \cdot t^n {D}_\om(f), \hbox{ and }\\
\left\|\tilde{\cL}_{\te^{-k_n(\om)}(\om)}^{k_n(\om)}(f) - \int f  d\nu_{\te^{-k_n(\om)}(\om)} \right\|_L^\om  & \leq & 4 B_\om \cdot
 t^n {D}_{\te^{-k_n(\om)}(\om)}(f).
\end{eqnarray*}
Furthermore, $\{\nu_\om: \om \in \Om\}$ is the unique measure with $\tilde{\cL}^\ast_{\om}(\nu_{\te \om})=\nu_\om$.
\end{theorem}

\noindent\textbf{Remark 1.} \label{ref:dual statements} Note that the second estimate in theorem \ref{theo:rpf-normalized} also could be reformulated in terms of $\cL_\om$ through identity (\ref{eq:identity_normalized_vs_nonnormalized}). Namely, it immediately follows from the proof that a.s., for each  Lipschitz continuous function $\{f_\om\}$,  
\begin{eqnarray*} 
 \left\| \frac{\cL_{\te^{-n}(\om)}^{n}(f)}{\La_{n}(\te^{-n}(\om)) h_{\te^{-n}(\om)}}  - \int f  d\mu_{\te^{-n}(\om)} \right\|_L^\om  & \leq & c^\ast_\om  K_{\te^{-n}(\om)} \cdot
 s^n \|f \|^{\te^{-n}(\om)}_L,
\end{eqnarray*} 
where $s$,and $K$ are given by theorem \ref{theo:rpf}. Moreover, there exists a random sequence $(k_n(\om):n \in \bbN)$ such that, with $t$ as in theorem \ref{theo:rpf} and using  $\alpha(\om) := 2B_\om$,
\begin{eqnarray*} 
 \left\| \frac{\cL_{\te^{-k_n}(\om)}^{k_n}(f)}{\La_{k_n}(\te^{-k_n}(\om)) h_{\te^{-k_n}(\om)}}  - \int f  d\mu_{\te^{-k_n}(\om)} \right\|_L^\om  & \leq & 2B_\om K_{\te^{-k_n}(\om)} \cdot
 t^n \|f \|^{\te^{-k_n}(\om)}_L.
\end{eqnarray*} 

\noindent\textbf{Remark 2.} \label{ref:precise estimates for K} The advantage in stating the exponential decay of correlations with respect to subsequences is that the constants are explicitly given, which could give rise to a refined analysis of stochastic stability of randomly perturbed intermittent maps as studied, e.g., in \cite{ShenStrien:2013}. Since the contraction in (vi) and (vii) stems from returns to a set with bounded parameters, $c$ can be chosen to be any value bigger than $\hbox{ess-inf}(2B_\om)$, even though the choice affects the sequences $(k_n)$ and $(l_n)$, which are essentially defined through consecutive visits to $\Om_{\hbox{\tiny bp}}$ and a subset of $\Om$ with bounded parameters.
 Furthermore, observe that the random variable $K_\om$ does not depend on the construction of $(k_n)$ and $(l_n)$. Namely, $K_\om$  
only depends on $B_\om$ and $\{h_\om \}$, that is, $ K_\om = 2  \| 1/h_\om \|_\infty \cdot \max\left\{ {\| h_\om \|_\infty }{\| 1/h_\om \|_\infty} -1 , B_\om -1,1 \right\}$ as shown in (\ref{def:k^ast}) below.

\noindent\textbf{Remark 3.} \label{ref:problem_with_proof}
We remark that there is an error in the proof of Theorem 4.1 in \cite{Stadlbauer:2010}
 which states that $\cL_\om^\ast(\mu_{\te\om}) = \la_\om \mu_\om$ and that $P_\om(s)/P_{\te \om}(s) \to \la_\om $ converges almost surely. The proof of the first statement is correct and follows the lines of the proof of Proposition 6.3 in \cite{DenkerKiferStadlbauer:2008}. However, since convergence in the narrow topology does not imply weak convergence almost surely, the proof of the second statement is not correct (see also the corrigendum in \cite{Stadlbauer:2014aa}). However, combining $\la_\om  = \int \cL_\om(1) d\mu_{\te \om}$ with (vii) in the above theorem implies that 
there exists a positive random variable $L$ such that, for a.e. $\om$, 
 \[ \| \cL_\om^{n+1}\varphi(1)/\cL_{\te\om}^{n}\varphi(1) -\la_\om  \|_\infty \leq L_\om s^n,\]
which is a significantly stronger statement than the convergence of $P_\om(s)/P_{\te \om}(s)$ to $\la_\om $. 

\subsection{Application to random positive matrices.} \label{subsec:random positive matrices}
For illustration of the theorem, we give an application to products of random positive matrices. Observe that a positive matrix always can be considered as a Ruelle operator with respect to a potential which is constant on cylinders of length 2. Or in other words, the potential is locally Hölder continuous with index 2 and $\kappa =0$, $B=1$. As in \cite{Stadlbauer:2010}, we consider random matrices
$A = \{A_\om \with \om \in \Om\}$ with  $A_\om=\big(p_{ij}^\om,\, i \in \cW^1_\om,j \in \cW^1_{\te\om}\big) $ and  $p_{ij}\geq 0$ a.s.. We then refer to $A$ as summable random matrix with the b.i.p.-property if
\begin{enumerate}
  \item the signum of $A$ defines a random topological Markov chain with the b.i.p.-property, 
  \item for a.e. $\om \in \te^{-1}(\Om_{\textrm{\tiny bi}} \cup \Om_{\textrm{\tiny bp}})$, 
  we have 
  \[\sup \left\{ {p_{ij}^\om}/{p_{ik}^\om} \with i \in \cW^1_\om,j,k \in \cW^1_{\te\om}, p_{ik}^\om \neq 0 \right\}< \infty,\] 
  \item the random variable $\om \mapsto  \sup_{j \in \cW^1_{\te\om}} |\log \sum_{i \in \cW^1_\om} p_{ij}^\om|$ is in  $L^1(P)$.
\end{enumerate}
As in \cite{Stadlbauer:2010}, it follows that theorem \ref{theo:rpf} is applicable and that the resulting eigenfunction and conformal measure are in fact vectors. Hence, for each summable random matrix $A$ with the b.i.p.-property, there exist a positive random variable $\la$ and strictly positive random vectors $h=\{h^\om \in \bbR^{\ell_\om - 1} \with \om \in \Om\}$ and $\mu=\{\mu^\om \in \bbR^{\ell_{\te\om} - 1} \with \om \in \Om\}$ such that, for a.e. $\om \in \Om$, we have $(h^\om)^t A_\om = \la_\om (h^{\te \om})^t$, $A_\om \mu^{\te \om} = \la_\om \mu^\om $. By applying (vii) of theorem \ref{theo:rpf} to the $i$-th unit vector $e_i := (\delta_{ij}: j \in \cW^1_\om)$, we obtain  
\begin{eqnarray*}
\sup_{ i \in \cW^1_\om, \ j \in \cW^1_{\te^n\om}} \left|  \frac{(A_\om^n)_{ij}}{\La_n(\om) h^{\te^n\om}_j} - \mu^\om_i  \right| 
\leq  c^\ast_\om  K_\om  \cdot s^n,
\end{eqnarray*}
with $\delta_{ij}$ referring to Kronecker's $\delta$ function and  $(\cdot)_{ij}$  to the coordinate $(i,j)$ of a matrix. As it easily can be seen, the above estimate implies exponential convergence of $ \La_n(\om)^{-1} A_\om^n$ in the sup-norm for matrices, that is,  
\[ \|\La_n(\om)^{-1} A_\om^n - \mu^\om \cdot (h^{\te^n \om})^t\|_\infty \leq   c^\ast_\om  K_\om (\sup_j h^{\te^n \om}_j)  s^n. \]
Let $\tilde{A}$ be the random matrix with constants $(\tilde{A}_\om)_{i,j} := ({A_\om})_{i,j} h^\om_i /(\la_\om h^{\te\om}_j)$. Since the matrix corresponds to Ruelle's operator with a potential which is constant on cylinders of length 2, we have that $B_\om=1$ for all $\om \in \Om$. By theorem \ref{theo:rpf-normalized}, it follows that there exists a  positive random vector $\{\nu^\om\}$
with $\sum_{i \in \cW^1_\om} \nu^\om_i =1$ a.s. such that, for all $n \in \bbN$,  
\begin{eqnarray*}
\sup_{ \genfrac{}{}{0pt}{}{i  \in \cW^1_\om,}{\ j \in \cW^1_{{\te^{l_n}(\om)}}}} \left| (\tilde{A}_\om^{l_n(\om)})_{ij} -\nu^\om_i  \right|   \leq   4 \cdot t^n, \; 
\sup_{ \genfrac{}{}{0pt}{}{i \in \cW^1_{\te^{-k_n}(\om)},}{j \in \cW^1_\om }} \left| (\tilde{A}_{\te^{-k_n}(\om)}^{{k_n}(\om)})_{ij} - \nu^{\te^{-k_n}(\om)}_i  \right|  
  \leq   4 \cdot t^n.
\end{eqnarray*} 
To obtain control on the sequences $l_n$ and $k_n$ and the parameter $t$, we have to impose further conditions on $\tilde{A}_\om$. For example, if $(\tilde{A}_\om )_{\mathbf{o}j} \geq C >0$ for all $\om \in \te^{-1}(\Om_{\hbox{\tiny bp} })$ and $j \in \cW^1_{\te \om}$, then $t = 1 - C/2$. Moreover, by setting $r < 1/2$, it follows that $ \lfloor -\log \alpha_\om/ \log r \rfloor =0$. By the proof of corollary \ref{eq:basic-decay-estimate}, the above is satisfied for the following  sequences $(l_n)$ and $(k_n)$: for $\om \in \Om$, set $l_1(\om):= \min\{n \geq 2: \te^n(\om) \in \Om_{\hbox{\tiny bp}}\}$ and by induction, for $n \in \bbN$,  
    \[ l_n(\om) :=  l_{n-1}(\om) + l_1(\te^{l_{n-1}(\om)}(\om)). \]     
Due to contraction only along passages from  $\te^{-1}(\Om_{\hbox{\tiny bp}})$ to $\Om_{\hbox{\tiny bp}}$, there is a slight asymmetry in the construction of $(k_n)$: Set $k_0(\om):= \min\{n \geq 1: \te^{-n}(\om) \in \Om_{\hbox{\tiny bp} }\}$  and $u(\om):= \min\{n \geq 2: \te^{-n}(\om) \in \Om_{\hbox{\tiny bp} }\}$. Then $(k_n)$ is defined by, for $n \in \bbN$,
\[ k_n(\om) := u(\te^{- k_{n-1}(\om)}(\om)) + k_{n-1}(\om). \]
For example, if $\Om_{\hbox{\tiny bp}} = \Om$, then $(l_n)$ and $(k_n)$ are the sequences $(2n: n \in \bbN)$. In this case, an argument as in Application \ref{ref:psi-mixing} below allows to substitute $(2n: n \in \bbN)$ by the sequence $(n: n \in \bbN)$.
 
These results are related to Problem 5.7 in \cite{Orey:1991a}: The transpose of $\tilde{A}_\om$ is a (random) stochastic matrix and conditions (i) and (ii) above imply that the iterates of $\tilde{A}^t_\om$ are contractions of the Wasserstein metric along subsequences which then implies exponential decay. However, this contraction property is a milder condition than $\delta^\ast =0$ in \cite{Orey:1991a}, which essentially means eventual contraction of the bounded variation norm.

\subsection{Application to decay of correlations.}
\label{ref:decay of correlations}
For $(X,T)$ with the b.i.p. property, $\hat{\hbox{H}}$ and $\cL_\om (1)=1$ a.s., 
theorem \ref{theo:rpf-normalized} also allows to deduce the following pathwise 
exponential decay of correlations using the fundamental identity
\[ \int f_\om \cdot g_{\te^n\om}\circ T^n_\om d\mu_\om =  \int \cL_\om^n(f_\om) \cdot  g_{\te^n\om} d\mu_{\te^n\om},\]
for a.e. $\om \in \Om$, $f$ Lipschitz and $g$ with $\int | g_{\te^n\om} |d\mu_{\te^n\om} < \infty$. 
Now assume that $\int f_\om d\mu_\om =0$. Then  
theorem \ref{theo:rpf-normalized}  implies that $ \|\cL_\om^{l_n(\om)}(f_\om)d\mu_\om)\|_\infty$ and 
$ \|\cL_{\te^{-k_n(\om)}\om}^{k_n(\om)}(f_{\te^{-k_n(\om)}\om}) d\mu_{\te^{-k_n(\om)}\om}\|_\infty$ converge to $0$ exponentially fast. Hence, for $f = \{f_\om\}$ Lipschitz continuous with $\int f d\mu_\om =0$ a.s. and  $g = \{g_\om\}$ with $\int |g |d\mu_\om < \infty$ a.s., we obtain the following estimates for the decay of correlations. 
\begin{eqnarray*}
\left| \int f  \cdot  g\circ T^{l_n(\om)}_\om d\mu_\om 
\right|
&\leq& 2 c \cdot t^n  D_\om(f) \int |g| d\mu_{\te^{l_n(\om)}\om},  \\
\left| \int f \cdot  g \circ T^{k_n(\om)}_{\te^{-k_n(\om)}\om} d\mu_{\te^{-k_n(\om)}\om} 
 \right|
&\leq& 4 B_\om \cdot t^n  D_{\te^{-k_n(\om)}\om}(f) \int |g| d\mu_{\om}.  \\ 
\end{eqnarray*}

\subsection{Application to $\psi$-mixing coefficients.}
\label{ref:psi-mixing}
As a corollary of theorem \ref{theo:rpf-normalized}, we obtain a mixing property by specifying the $\psi$-mixing coefficients known from probability theory. In case of a random topological Markov chain, these are defined by
\[ \psi_n(\omega) := \sup \frac{\mu_{\te^{-k}\om}([a] \cap T_{\te^{-k}\om}^{-k - n}(A)) - \mu_{\te^{-k}\om}([a])\mu_{\te^{n}\omega}(A)}{\mu_{\te^{-k}\om}([a]) \mu_{\te^{n}\omega}(A)}, \]
where the supremum is taken with respect to all $k \in \bbN$, $a \in \cW^k_{\te^{-k}\om}$ and  $A \subset X_{\te^{n}\omega}$ measurable such that  $\mu_{\te^{-k}\om}([a]),\mu_{\te^{n}\omega}(A) >0$.  
 Observe that, for $a \in \cW^k$ with $P(\Om_a)>0$ and 
\[f_a^\om(x) := \1_{\te^k(\Om_a))}(\om)  e^{S_k\varphi_{\te^{-k}\om} \circ \tau_a(x)},\]
the estimate (\ref{eq:key}) implies that  $f_a^\om$ is Lipschitz continuous and, in particular, that  
\[  D_\om(f_a^\om) \leq   \frac{B_\om^2}{  \mu_\om (T_{\te^{-k}\om}([a]))}   \mu_{\te^{-k}\om}([a]). \]
Now assume that $(X,T)$ satisfies the b.i.p.-property, $\hat{\hbox{H}}$ and $\cL_\om (1)=1$ a.s.. Then $\{\mu_\om \}$ is a random invariant measure and theorem \ref{theo:rpf-normalized} implies that, for $n \in \bbN$ and $A \subset X_{\te^{l_n(\om)}\omega}$ measurable,
\begin{align*}
  & \left| \mu_{\te^{-k}\om}([a] \cap T_{\te^{-k}\om}^{-k - l_n(\om)}(A)) - \mu_{\te^{-k}\om}([a])\mu_{\te^{l_n(\om)}\omega}(A) \right| \\
 =  &  \left| \int f_a^\om \1_A \circ T_\om^{l_n(\om)} d\mu_\om - \int \mu_{\te^{-k}\om}([a])\1_A \circ T_\om^{l_n(\om)} d\mu_\om \right| \\
 = &  \left| \int  \left(\cL^{l_n(\om)}_\om(f_a^\om) - \mu_{\te^{-k}\om}([a])\right) \1_A d\mu_{\te^{l_n(\om)}\om}  \right|
 \leq    2c  \mu_{\te^{l_n(\om)}\omega}(A)  D_\om(f_a^\om)\\
  \leq  & \frac{2 c B_\om^2}{  \mu_\om (T_{\te^{-k}\om}([a]))} t^n \mu_{\te^{-k}\om}([a]) \mu_{\te^{l_n(\om)}\omega}(A).
\end{align*}
In particular, if $\om \in  \Om_{\textrm{\tiny bi}}$, then $\mu_\om (T_{\te^{-k}\om}([a])$ is bounded from below. Hence, there exists  $\Om' \subset \Om_{\textrm{\tiny bi}}$ of positive measure and $C>0$ such that, for all $\om \in \Om'$, 
\begin{equation}\label{eq:psi-mixing2}  
\psi_{l_n}(\om) \leq C t^n 
\end{equation}
Furthermore, recall that the random sequence $l_n$ is constructed through not necessarily first returns as follows. Let $\Om_{B,C}$ be as in the proof corollary 
 \ref{cor:Markov_Version}, that is $B_\om$ and $C_\om$ (see (\ref{eq:bound from below - main proof})) are uniformly bounded on $\Om_{B,C}$. The sequence $(l_n)$ then has to be chosen such that $\te^{l_n(\om)}(\om) \in \Om_{B,C}$ for all $n \in \bbN$ and such that the difference between two consecutive elements of $(l_n)$ is bounded from below by the random variable $L(\om):= m_\om + n_{\te^{m_\om}(\om)}$ (see lemma \ref{prop:main_estimates}), that is 
 \[ l_n(\om)   - l_{n-1}(\om) \geq L(\te^{l_{n-1}(\om)}(\om)). \]
In the following, $(l_n)$ is specified in case of a random full shift, $B_\om$ uniformly bounded from above and $\phi_\om|_{[\mathbf{o}_\om]}$ uniformly bounded from below. Under these conditions, $K := \hbox{esssup}\{ m_\om + n_{\te^{m_\om}(\om)}\} < \infty$ and, in particular, (\ref{eq:psi-mixing2}) holds for $(l_n) = (Kn + k : n \in \bbN)$, for each $k=0,\ldots K-1$. Hence, we have that, for all $\om \in \Om'$ and with $\tilde{C}:= C/t$ and $\tilde{t} := \sqrt[K]{t}$ and for all $n \geq K$,
\begin{equation}\nonumber \label{eq:psi-mixing3} 
\psi_{n}(\om) \leq \tilde{C} \tilde{t}^n.
\end{equation}

\subsection{Application to equilibrium states.} Under the assumptions of theorem \ref{theo:rpf}, it was shown in  \cite{DenkerKiferStadlbauer:2008} that a variational principle holds. That is, 
\[ P_G(\varphi) = \sup \left\{ h_m^{\hbox{\scriptsize (r)}}(T) + \int \varphi dm_\om dP \with m \in \mathcal{M}_\te(T) \right\},  \]
where $h_m^{\hbox{\tiny (r)}}(T)$ refers to the fiber entropy as defined in \cite{KiferLiu:2006a} and $\mathcal{M}_\te(T)$ to the random invariant probability measures, that is to those random probability measures $m=\{m_\om\}$ such that $m_\om \circ T_\om^{-1}= m_{\te \om} $ a.s.. 
For the random probability measure $\nu = \{\nu_\om\}$ defined by $d\nu_\om = h_\om d\mu_\om$, with $\{h_\om\}$ and  $\{\mu_\om\}$ given by theorem \ref{theo:rpf}, it then follows that 
\[ P_G(\varphi) = h_\nu^{\hbox{\scriptsize (r)}}(T) + \int \varphi_\om d\nu_\om dP, \]
that is, $\{\nu_\om\}$ is an equilibrium state. In order to prove this assertion, recall from  \cite{KiferLiu:2006a} that
\[h_m^{\hbox{\tiny (r)}}(T) = - \lim_{n \to \infty} \sum_{a \in \cW^n_\om} m_\om([a]) \log m_\om([a]). \] 
By the Gibbs property in \cite[Remark 4.2]{Stadlbauer:2010} (as in (\ref{eq:Gibbs-property}) below), there exists a subset $\Om'$ of positive measure of $\Om_{\hbox{\tiny bi}}$ and  $F>0$ such  that $ \mu_{\om}([a]_\om) = F^{\pm 1} \exp (S_n \varphi_\om(x))/\Lambda_n(\om)$ for all $n \in \bbN$ with $\te^n \om \in \Om'$, $a \in \cW^n_\om$ and $x \in [a]_\om$. For $d\nu_\om = h_\om d\mu_\om$, we hence have by 
invariance of $\{\nu_\om\}$, Hölder continuity of $h_\om$ and Birkhoff's theorem that
\begin{eqnarray*}
h_\mu^{\hbox{\scriptsize (r)}}(T) &=&  \lim_{n \to \infty, \te^n \om \in \Om'} - \frac{1}{n} \sum_{a \in \cW^n_\om}  \nu_{\om}([a]_\om)  \log \nu_{\om}([a]_\om) \\
& =&  \lim_{n \to \infty, \te^n \om \in \Om'} - \frac{1}{n} \int_{X_\om} \log h_\om  + S_n \varphi_\om - \log\Lambda_n(\om) d\nu_\om \\
& =&  \lim_{n \to \infty, \te^n \om \in \Om'} \frac{1}{n} \left( \int_{X_\om} \log h_\om  d\nu_\om +  \sum_{k=0}^{n-1} \left(\log  \lambda_{\te^k \om} -  \int_{X_\om} \varphi_{\te^k \om}\circ T_\om^k  d\nu_\om \right) \right)\\
&=& P_G(\varphi) - \int  \varphi_\om d\nu_\om dP.
\end{eqnarray*} 
Hence, $\{\nu_\om\}$ is an equilibrium state for $(X,T)$.

\section{Proofs of theorems \ref{theo:rpf} and \ref{theo:rpf-normalized}} \label{sec:4}

This section is exclusively devoted to the proofs of  theorems \ref{theo:rpf} and \ref{theo:rpf-normalized}. Observe that parts (iii) and (iv) of theorem \ref{theo:rpf} are as in \cite[Th. 4.2]{Stadlbauer:2010} and hence we assume that $\lambda$ and $\{\mu_\om\}$ are given. However, in order to control the regularity of $\{h_\om\}$, we employ a very similar construction to the one in \cite{DenkerKiferStadlbauer:2008}. 

\medskip
\noindent {\em (1) Construction of $h_\om$.}  For $\om \in \Om$ and $k \in \bbN$, let
$ f_{\om,k} := \La_k(\te^{-k}\om)^{-1} \cL^k_{\te^{-k}\om}(1)$. By the same argument as in the proof of proposition \ref{prop:L_leaves_invariant_the_Lipschitz_functions}, we have, for  $x,y \in [a]_{\om}$ for some $a \in \cW^1_{\om}$, that
\[   |f_{\om,k}(x) - f_{\om,k}(y)| \leq  
\frac{\cL_{\te^{-k}\om}^k(1)(y) (B_\om - 1)d_r(x,y)}{\La_k(\te^{-k}\om)} = f_{\om,k}(y)   (B_\om - 1)d_r(x,y).\] 
It follows from this and $\int \log B_\om dP < \infty$, that $\log f_{\om,k}$ is locally H\"older continuous with index 1 with Hölder constant bounded by $B_\om - 1$. Moreover, by remark 4.2 in \cite{Stadlbauer:2010}, the measure $\{\mu_\om\}$ satisfies the following Gibbs property.
That is, for a.e. $\om \in \Om_{\hbox{\tiny bi}}$, $k \in \bbN$ and $a \in \cW^k_{\te^{-k}\om}$ and $x \in [a]_{\te^{-k}\om}$, 
\begin{equation} \label{eq:Gibbs-property}  
  (\La_k(\te^{-k}\om))^{-1} e^{S_k \phi_{\te^{-k}\om}(x)} \leq    B_{\om} 
  (E_\om)^{-1}  \mu_{\te^{-k}\om}([a]_{\te^{-k}\om}),
 \end{equation}
where $E_{\om} :=  \min_{b \in \cI}\mu_\om(T_{\te^{-1}\om}([b]_{\te^{-1}\om}))$ is strictly positive by finiteness of $\cI$. By summing over all cylinders, we obtain that $f_{\om,k}(x) \leq B_{\om} (E_\om)^{-1}$ for a.e.  $\om \in \Om_{\hbox{\tiny bi}}$ and $k \in \bbN$. Moreover, Corollary 4.1 in \cite{Stadlbauer:2010} shows that there exists $\Om^\ast \subset \Om$ of positive measure such that $f_{\om,k}(x)$ is bounded from below whenever $\te^{-k}\om \in \Om^\ast$. Hence, the random function $(h_\om)$ defined by
\[ h_\om(x) := \liminf_{k \to \infty, \te^{-k}\om \in \Om^\ast} f_{\om,k} (x)\]
is bounded from above and below.

\medskip
\noindent {\em (2) Properties of $h_\om$.} 
Observe that Fatou's Lemma implies that $\cL_\om(h_\om) \leq \la_\om h_{\te \om}$. Hence, for $\om \in \Om$ and $n \in \bbN$ with $\te^n \om \in \Om_{\hbox{\tiny bi}}$, we have
\[ \int h_\om d\mu_\om =  \int \La_n(\om)^{-1} \cL^n_\om(h_\om) d\mu_{\te^n\om} \leq \int h_{\te^n\om} d\mu_{\te^n\om}  \leq  B_{{\te^n\om}} (E_{\te^n\om})^{-1}.\]
By ergodicity of $\te$, it  follows that $\int h_\om d\mu_\om \leq \hbox{ess-inf}_{\om \in \Om_{\hbox{\tiny bi}}} B_{\om} (E_{\om})^{-1}$. By the same arguments as in \cite{DenkerKiferStadlbauer:2008}, it then follows again from ergodicity that $\cL_\om(f_\om) = \la_\om f_{\te \om}$. It also follows from the above estimates that $\{\log h_\om \}$ is  locally H\"older continuous with index 1 and Hölder constant $B_\om - 1$ and that $\|h_\om\|_\infty \leq B_{\om} (E_\om)^{-1}$ for a.e. $\om \in \Om_{\hbox{\tiny bi}}$. In order to show that $h_\om$ is bounded from below, note that finiteness of $\cI$ and local H\"older continuity imply that, for a.e. $\om \in \Om_{\textrm{\tiny bi}}$ and $x \in X_\om$, 
\[ \la_{\te^{-1}\om} h_\om(x) = \cL_{\te^{-1}\om}(h_{{\te^{-1}\om}})(x) \geq \min_{a \in \cI} \left(\inf_{y \in [a]_{\te^{-1}\om}} e^{\varphi_{\te^{-1} \om}(y) } h_{\te^{-1}\om}(y)\right) >0. \]
Hence, $\inf h_\om >0$ for a.e. $\om \in \Om_{\textrm{\tiny bi}}$. Combining these upper and lower bounds with  positivity of $\cL_\om$ gives that $0< \inf h_\om \leq \sup h_\om < \infty$ for a.e. $\om \in \Om$. This proves (i) and (ii) of theorem \ref{theo:rpf}.

\medskip
\noindent {\em (3) Normalizing the operator.}  Part (ii) implies that  $\log h_\om(x)$ is locally H\"older with index 1 and, in particular, that $\log h_{\te\om}(T_\om(x))$ is locally  H\"older with index 2.  Moreover, since $h_\om(x)$ is bounded from above and below, $V_1^\om (\log h_{\te\om}\circ T_\om) < \infty$ almost surely. In particular, the normalised potential defined by
\[\tilde{\varphi}_\om (x):=  \varphi_\om(x) + \log h_\om(x) - \log h_{\te\om}(T_\om(x)) - \log \lambda_\om\]
satisfies property  ($\hat{\hbox{H}}$). Furthermore, we have  
\begin{equation}\label{eq:identity_normalized_vs_nonnormalized}
\La_n(\om) h_{\te^n\om} \tilde{\cL}^n_\om(f) =  {\cL}^n_\om(f \cdot h_\om)
\end{equation}
for a.e. $\om$, all $n \in \bbN$ and each bounded and continuous function $f:X_\om \to \bbR$. In particular,  $\tilde{\cL}_\om(1)=1$ and, 
for $\{\nu_\om\}$ defined by $d \nu_\om := h_\om d\mu_\om$, we have $ \tilde{\cL}^\ast_{\om}(\nu_{\te \om})=\nu_\om$.
Now let $t \in (0,1)$, $c >0 $ and $(l_n(\om))$ be given by corollary \ref{cor:Markov_Version}. 
For a fibrewise Lipschitz function $f$ with $D(f)\leq 1$, $n \in \bbN$ and $x \in X_{\te^{l_n}(\om)}$, the Monge-Kantorovich duality and (i) in corollary \ref{cor:Markov_Version} then imply that
\begin{eqnarray} \label{eq:basic-decay-estimate}
 \left| \tilde{\cL}^{l_n}_\om(f)(x)  - \int f d\nu_\om \right| &=&  \left| \int f d (\tilde{\cL}_\om^{l_n})^\ast(\delta_x)  - \int f d (\tilde{\cL}_\om^{l_n})^\ast(\nu_{\te^{l_n}\om}) \right|\\
\nonumber
 & \leq & W((\tilde{\cL}_\om^{l_n})^\ast(\delta_x) , (\tilde{\cL}_\om^{l_n})^\ast(\nu_{\te^{l_n}\om}) ) 
  \leq  c t^n  W(\delta_x , \nu_\om ) \\
\nonumber
  &= & c \cdot t^n \int d(x,y) d \nu_\om(y) \leq ct^n. 
\end{eqnarray} 
Hence, $ \| \tilde{\cL}^{l_n}_\om(f)  - \int f d\nu_\om \|_\infty \leq c t^n$. Combining this with (ii) of corollary \ref{cor:Markov_Version} then proves the first inequality in theorem \ref{theo:rpf-normalized} and the second follows by the same arguments. Finally, the uniqueness of $\{\nu_\om \}$ clearly follows from (i) in corollary \ref{cor:Markov_Version}. This proves theorem \ref{theo:rpf-normalized}. 

\medskip
\noindent {\em (4) Transferring the results.} By (\ref{eq:identity_normalized_vs_nonnormalized}), we have
\begin{equation}\label{eq:identity-normal-vsnonnormal-2}
 \frac{\cL_\om^{n} (f)}{ \La_{n}(\om) h_{\te^{n}(\om)} } - \int f d\mu_\om = 
 \tilde{\cL}_\om^{n} (f/h) - \int (f/h)   d\nu_\om. \end{equation}
Hence, parts (vi) and (vii) in theorem \ref{theo:rpf} can be proved through estimating $D_\om(f/h)$. By dividing the supremum below into $x,y$ according to $d(x,y)=1$ or $d(x,y)<1$ and then applying locally H\"older continuity with index 1 of $\log h$ in the second case, it follows that  
\begin{eqnarray}
\label{def:k^ast} 
 D_\om(f/h) &\leq & \| 1/h_\om \|_\infty D_\om(f) + \|f_\om\|_\infty D(1/h_\om) \\
\nonumber
  &\leq & \| 1/h_\om \|_\infty \left( D_\om(f) + \|f_\om\|_\infty \sup_{x,y \in X_\om} \frac{|h_\om(x)/h_\om(y) -1|}{d(x,y)}   \right)\\
 \nonumber  &\leq &   \|f_\om\|_L  \| {\textstyle\frac{1}{h_\om}} \|_\infty
 \max\left\{ 1, {\| h_\om \|_\infty }{\| {\textstyle\frac{1}{h_\om}} \|_\infty}-1,  
 B_\om -1
  \right\} =:  \|f_\om\|_L K_\om/2
\end{eqnarray}
Combining the estimate with (\ref{eq:identity-normal-vsnonnormal-2}) and theorem \ref{theo:rpf-normalized}  proves (vi) of theorem \ref{theo:rpf}.
We will proceed with the proof of (vii). In order to do so, note that corollary \ref{cor:Absolute_Markov_Version} implies, using  (\ref{eq:basic-decay-estimate}), that 
\begin{eqnarray*}
\left| \tilde{\cL}^{n}_{\om}(f)(x)  - \int f d\nu_{\om} \right| 
 & \leq &
  c^\ast_\om s^n  W(\delta_x , \nu_{\te^{n}\om} ) \leq   c^\ast_\om s^n,   \\
 \left|  \tilde{\cL}^{n}_{\om}(f)(x)  -  \tilde{\cL}^{n}_{\om}(f)(y)  \right| 
 & \leq &
  c^\ast_\om s^n  W(\delta_x , \delta_y ) =  c^\ast_\om s^n  d(x,y),  
 \end{eqnarray*}
 where $f$ satisfies $D(f)\leq 1$. Combining (\ref{eq:identity-normal-vsnonnormal-2}) with the estimate on $D(f/h)$ proves (vii).

\medskip
\noindent {\em (5) Uniqueness.} Hence, in order to prove theorem \ref{theo:rpf}, it remains to deduce the uniqueness of $\{h_\om\}$ and $\{\mu_\om\}$. So assume that $\{\mu^{(i)}_\om\}$ with $i=1,2$ are random probability measures with $\cL_\om^\ast(\mu^{(i)}_{\te\om}) = \la(\om)\mu^{(i)}_\om$. Then $\{h_\om d\mu^{(i)}_\om\}$ are both $\tilde{\cL}_\om$- invariant and hence equal by uniqueness from theorem \ref{theo:rpf-normalized}. The uniqueness of  $\{h_\om d\mu^{(i)}_\om\}$ follows from the same arguments. Hence, also theorem \ref{theo:rpf} is proved.

\subsection*{Acknowledgements} 
The author would like to express his  gratitude to the Erwin Schrödinger Institute in Vienna for warm hospitality and excellent working conditions and the DFG-network ``Skew product dynamics and multifractal analysis'' for financial support: The ideas of proofs emerged while preparing and presenting a mini course at the Workshop on complexity and dimension theory of skew product systems. The author also acknowledges support by \emph{Fundação para Ciência e a Tecnologia} through project PTDC/MAT/120346/2010.

\end{document}